\documentclass[11pt, reqno]{amsart}
\usepackage[margin=1.1in]{geometry}
\usepackage{amsmath,amsthm,amssymb,enumerate,fancyhdr}
\usepackage{hyperref} 
\hypersetup{colorlinks=true, linkcolor=blue, citecolor=blue, filecolor=blue, urlcolor=blue}
\usepackage[comma, sort&compress, numbers]{natbib}
\usepackage [latin1]{inputenc}

\newcommand{\prob}{\stackrel{P}{\longrightarrow}}

\newcommand{\one}{{\bf 1}}

\newcommand{\reals}{{\mathbb R}}
\newcommand{\bbr}{\reals}
\newcommand{\cal}{\mathcal}
\newcommand{\vep}{\varepsilon}

\newcommand{\bbg}{\protect{\mathbb G}}

\newcommand{\pp}{{\prime\prime}}

\newtheorem{theorem}{Theorem}[section]

\newtheorem{fact}{Fact}[section]
\newtheorem{lemma}{Lemma}[section]
\newtheorem*{defn}{Definition}

\newtheorem{remark}{Remark}[section]

\def\Cov{{\rm Cov}}

\def\Var{{\rm Var}}
\def\E{{\rm E}}

\def\diag{{\rm Diag}}

\DeclareMathOperator{\Tr}{Tr}

\numberwithin{equation}{section}

\begin{document}

\title[Largest eigenvalue of inhomogeneous ERRG]{Eigenvalues outside the bulk of inhomogeneous Erd\H{o}s-R\'enyi random graphs}
\author[A. Chakrabarty, S. Chakraborty and R.~S.~Hazra]{Arijit Chakrabarty, Sukrit Chakraborty and Rajat Subhra Hazra}
\address{Indian Statistical Institute\\ 203 B.T. Road, Kolkata, 700108\\ India}
\email{arijit.isi@gmail.com, sukrit049@gmail.com and rajatmaths@gmail.com}

\newcommand{\acr}{\newline\indent}

\begin{abstract}
In this article, an inhomogeneous Erd\H{o}s-R\'enyi random graph on $\{1,\ldots, N\}$ is considered, where an edge is placed between vertices $i$ and $j$ with probability $\vep_N f(i/N,j/N)$, for $i\le j$, the choice being made independently for each pair. The integral operator $I_f$ associated with the bounded function $f$ is assumed to be symmetric, non-negative definite, and of finite rank $k$. We study the edge of the spectrum of the adjacency matrix of such an inhomogeneous Erd\H{o}s-R\'enyi random graph under the assumption that $N\vep_N\to \infty$ sufficiently fast. Although the bulk of the spectrum of the adjacency matrix, scaled by $\sqrt{N\vep_N}$, is compactly supported, the $k$-th largest eigenvalue goes to infinity. It turns out that the largest eigenvalue after appropriate scaling and centering converges to a Gaussian law, if the largest eigenvalue of $I_f$ has multiplicity $1$. If $I_f$ has $k$ distinct non-zero eigenvalues, then the joint distribution of the $k$ largest eigenvalues converge jointly to a multivariate Gaussian law. The first order behaviour of the eigenvectors is derived as a byproduct of the above results. The results complement the homogeneous case derived by \cite{Erdos1}. 
\end{abstract}

\maketitle

\section{Introduction}\label{sec:intro} 
Given a graph on $N$ vertices, say, $\{1,\ldots, N\}$, let $A_N$ denote the adjacency matrix of the graph, whose $(i,j)$-th entry is $1$ if there is an edge between vertices $i$ and $j$ and $0$ otherwise. Important statistics of the graph are the eigenvalues and eigenvectors of $A_N$ which encode crucial information about the graph. The present article considers the generalization of the most studied random graph, namely the Erd\H{o}s--R\'enyi random graph (ERRG). It is a graph on $N$ vertices where an edge is present independently with probability $\vep_N$. The adjacency matrix of the ERRG is a symmetric matrix with diagonal entries zero, and the entries above the diagonal are independent and identically distributed Bernoulli random variables with parameter $\vep_N$. We consider an inhomogeneous extension of the ERRG where the presence of an edge between vertices $i$ and $j$ is given by a Bernoulli random variable with parameter $p_{i,j}$ and these $\{p_{i,j}:\, 1\le i< j\le N\}$ need not be same. When $p_{i,j}$ are same for all vertices $i$ and $j$ it shall be referred as (homogeneous) ERRG. 

 The mathematical foundations of inhomogeneous ERRG where the connection probabilities $p_{i,j}$ come from a discretization of a symmetric, non-negative function $f$ on $[0,1]^2$ was initiated in \cite{bollobas2007phase}. The said article considered edge probabilities given by
\[
p_{i,j}= \frac{1}{N} f\left( \frac{i}{N},\, \frac{j}{N}\right).
\]
In that case the average degree is bounded and the phase transition picture on the largest cluster size was studied in the same article (see also \cite{van2013, bhamidi2010} for results on inhomogeneous ERRG).
The present article considers a similar set-up where the average degree is unbounded and studies the properties of eigenvalues of the adjacency matrix. The connection probabilities are given by
$$p_{i,j}=\vep_N f\left(\frac{i}{N}, \frac{j}{N}\right)$$
with the assumption that 
\begin{equation}\label{eq:non-sparse}
N\vep_N\to\infty.
\end{equation}

Let $\lambda_1(A_N)\ge \ldots \ge \lambda_N(A_N)$ be the eigenvalues of $A_N$. It was shown in \cite{chakrabarty2018spectra} (see also \cite{zhu:2018} for a graphon approach) that the  empirical distribution of the centered adjacency matrix converges, after scaling with $\sqrt{N\vep_N}$, to a compactly supported measure $\mu_f$. When $f\equiv 1$, the limiting law $\mu_f$ turns out to be the semicircle law. Note that $f\equiv 1$ corresponds to the (homogeneous) ERRG (see \cite{TranVuWang1, ding2010spectral} also for the homogeneous case). Quantitative estimates on the largest eigenvalue of the homogeneous case (when $N\vep_N\gg (\log N)^4$) were studied in \cite{komlos1981eigen,vu2007spec} and it follows from their work that the smallest and second largest eigenvalue converge to the edge of the support of semicircular law. The results were improved recently in \cite{Benaych2017spectral} and the condition on sparsity can be extended to the case $N\vep_N\gg \log N$ (which is also the connectivity threshold). It was shown that inhomogeneous ERRG also has a similar behaviour. The largest eigenvalue of inhomogeneous ERRG when $N\vep_N\ll \log N$ was  treated in \cite{benaych2019largest}. Under the assumption that $N^{\xi}\ll N\vep_N$ for some $\xi\in (2/3, 1]$, it was proved in \cite[Theorem 2.7]{Erdos2}  that the second largest eigenvalue of the (homogeneous) ERRG after appropriate centering and scaling converges in distribution to the Tracy-Widom law. The results were recently improved in \cite{lee2018}. The properties of the largest eigenvector in the homogeneous case was studied in \cite{TranVuWang1, Erdos2, lee2018, MR4089498, alt:ducatez:knowles:2020}. 

The scaling limit of the maximum eigenvalue of inhomogeneous ERRG also turns out to be interesting. The fluctuations of the maximum eigenvalue in the  homogeneous case were studied in \cite{Erdos1}. It was proved that
$$ (\vep_N(1-\vep_N))^{-1/2}\left(\lambda_1( A_N)- \E[\lambda_1(A_N)]\right) \Rightarrow N(0, 2).$$
The above result was shown under the assumption that 
\begin{equation}\label{eq.vep}
(\log N)^{\xi} \ll N\vep_N
\end{equation}
for some $\xi>8$, which is a stronger assumption than \eqref{eq:non-sparse}. 

It is well known that in the classical case of a (standard) Wigner matrix, the largest eigenvalue converges to the Tracy-Widom law. We note that there is a different scaling between the edge and bulk of the spectrum in ERRG. As pointed out before,  the bulk is of the order $(N\vep_N)^{1/2}$ and the order of the largest eigenvalue is $N\vep_N$. Letting
\begin{equation}
\label{eq.defw}W_N=A_N-\E(A_N)\,,
\end{equation}
where $\E(A_N)$ is the entrywise expectation of $A_N$, it is easy to see that
\[
A_N=\vep_N\one\one^\prime+W_N\,,
\]
where $\one$ is the $N\times1$ vector with each entry $1$. Since the empirical spectral distribution of $(N\vep_N)^{-1/2}W_N$ converges to semi-circle law, the largest eigenvalue of the same converges to $2$ almost surely.  
As $\E[A_N]$ is a rank-one matrix, it turns out that the largest eigenvalue of $A_N$ scales like $N\vep_N$, which is different from the bulk scaling. 

The above behaviour can be treated as a special case of the perturbation of a Wigner matrix. When $W_N$ is a symmetric random matrix with independent and identically distributed entries with mean zero and finite variance $\sigma^2$ and the deformation is of the form $$M_N= \frac{W_N}{\sqrt{N}}+ P_N,$$
the largest eigenvalue is well-studied. Motivated by the study of adjacency matrix of homogeneous ERRG, \citet{furedi:kolmos:1981} studied the above deformation with $P_N= m N^{-1/2} \one \one^\prime$, $m\neq 0$. They showed that
$$N^{-1/2}\left( \lambda_1(M_N)- Nm-\frac{\sigma^2}{m}\right)\Rightarrow N(0, 2\sigma^2).$$ Since the bulk of $N^{-1/2} M_N$ lies within $[-2\sigma, \, 2\sigma]$, the largest eigenvalue is detached from the bulk. In general, when the largest eigenvalue of the perturbation has the same order as that of the maximum eigenvalue of $W_N$, the problem is more challenging. One of the seminal results in this direction was obtained in \citet{BBP:2005}. They exhibited a phase transition in the behaviour of the largest eigenvalue for complex Wishart matrix, which is now referred to as the BBP (Baik-Ben Arous- P\'ech\'e) phase transition. It is roughly as follows. Suppose $P_N$ is a deterministic matrix of rank $k$ with non-trivial eigenvalues $\theta_1\ge \theta_2\dots\ge \theta_k>0$. If $\theta_i\le \sigma$, then $\lambda_i(M_N)\to 2\sigma$ almost surely, and if $\theta_i> \sigma$ then
$$\lambda_i\to \theta_i+\frac{\sigma^2}{\theta_i},\,\,  \text{ almost surely}.$$
See \cite{feral:peche, baik:silverstein} for further extensions. It is clear that when $\theta_i>\sigma$ the corresponding eigenvalue lies outside the bulk of the spectrum. The phase transition is also present at the level of fluctuations around $2\sigma$ or $\theta_i+\sigma^2/\theta_i$. It is known that under some moment conditions on the entries of $W_N$ (see \cite{capitaine2009largest, knowles:yin:2013, knowles:yin:2014}), when $\theta_i\le \sigma$, the fluctuations are governed by the Tracy-Widom law, and when $\theta_i>\sigma$, the limiting distribution is given by the eigenvalues of a random matrix of order $k$. This limiting random matrix depends on the eigenvectors of $P_N$ and also on the entries of $W_N$. The non-universal nature was pointed out in \cite{capitaine2009largest}. For example, when $W_N$ is a Gaussian orthogonal ensemble and $P_N=\theta_1 \one\one^{\prime}$ then the limit is Gaussian and if the entries of $W_N$ are not from a Gaussian distribution, then the limit is a convolution of Gaussian and the distribution from which the entries of $W_N$ are sampled. One can find further details in \cite{BBP:2005, benaych:2011, capitaine2009largest, capitaine2012central, knowles:yin:2013, knowles:yin:2014, FBG:Nadakuditi} and the survey by \cite{peche:review}. The case when the rank $k$ depends on $N$ was considered in \cite{capitaine:peche,johansson, lee:schnelli}. Various applications of these results on outliers can be found in the literature, for example,  \cite{ali:couillet, chapon:hachem, coulliet:hachem}.

The adjacency matrix of the inhomogeneous ERRG does not fall directly into purview of the above results, since  $W_N$, as in \eqref{eq.defw}, is a symmetric matrix, with independent entries above the diagonal, but the entries have a variance profile, which also depends on the size of the graph. The inhomogeneity does not allow the use of local laws suitable for semicircle law in an obvious way. The present article aims at extending the results obtained in \cite{Erdos1} for the case that $f$ is a constant to the case that $f$ is a non-negative, symmetric, bounded, Riemann integrable function on $[0,1]^2$ which induces an integral operator of finite rank $k$, under the assumption that \eqref{eq.vep} holds.  The case $k\ge2$ turns out to be substantially difficult than the case $k=1$ for the following reason. If $k=1$, that is, 
\[
\E(A_N)=u_Nu_N^\prime\,,
\]
for some $N\times1$ deterministic column vector $u_N$, then with high probability it holds that
\[
u_N^\prime\left(\lambda I-W_N\right)^{-1}u_N=1\,,
\]
where $\lambda$ is the largest eigenvalue of $A_N$. The above equation facilitates the asymptotic study of $\lambda$. However, when $k\ge2$, the above equation takes a complicated form. The observation which provides a way out of this is that $\lambda$ is also an eigenvalue of a $k\times k$ matrix with high probability; the same is recorded in Lemma \ref{proof.l1} of Section \ref{sec:proof}. Besides, working with the eigenvalues of a $k\times k$ matrix needs more linear algebraic work when $k\ge2$. For example, the proof of Lemma \ref{proof.l6}, which is one of the major steps in the proof of a main result, becomes a tautology when $k=1$.

The following results are obtained in the current paper. If the largest eigenvalue of the integral operator has multiplicity $1$, then the largest eigenvalue of the adjacency matrix has a Gaussian fluctuation. More generally, it is shown that the eigenvalues which correspond to isolated eigenvalues, which will be defined later, of the induced integral operator jointly converge to a multivariate Gaussian law. Under the assumption that the function $f$ is Lipschitz continuous, the leading order term in the expansion of the expected value of the isolated eigenvalues is obtained. Furthermore, under an additional assumption, the inner product of the eigenvector with the discretized eigenfunction of the integral operator corresponding to the other eigenvalues is shown to have a Gaussian fluctuation. Some important examples of such $f$ include the rank-one case, and the stochastic block models. It remains an open question to see if the $(k+1)$-th eigenvalue follows a Tracy-Widom type scaling. 

The mathematical set-up and the main results of the paper are stated in Section \ref{sec:result}. Theorem \ref{t.main} shows that of the $k$ largest eigenvalues, the isolated ones, centred by their mean and appropriately scaled, converge to a multivariate normal distribution. Theorem \ref{t.mean} studies the first and second order of the expectation of the top $k$ isolated eigenvalues. Theorems \ref{t.eigvec} and \ref{t.flac} study the behaviour of the eigenvectors corresponding to the top $k$ isolated eigenvalues. Section \ref{sec:eg} contains the special case when $f$ is rank one and the example of stochastic block models.  A few preparatory estimates are noted in Section \ref{sec:est}, which are used later in the proofs of the main results, given in Section \ref{sec:proof}. The estimates in Section \ref{sec:est} are proved in Section \ref{sec:app}.

\section{The set-up and the results}\label{sec:result}
Let $f:[0,1]\times[0,1]\to[0,\infty)$ be a function which is symmetric, bounded, and Riemann integrable, that is, 
\begin{equation}
\label{eq.sym}f(x,y)=f(y,x)\,,0\le x,y\le1\,,
\end{equation}
and the set of discontinuities of $f$ in $[0,1]\times[0,1]$ has Lebesgue measure zero. 

The integral operator $I_f$ with kernel $f$ is defined from $L^2[0,1]$ to itself by
\[
\bigl(I_f(g)\bigr)(x)=\int_0^1f(x,y)g(y)\,dy\,,\,0\le x\le1\,.
\]
Besides the above, we assume that $I_f$ is a non-negative definite operator and the range of $I_f$ has a finite dimension.

Under the above assumptions $I_f$ turns out to be a compact self-adjoint operator, and from the spectral theory one obtains 
$\theta_1\ge\theta_2\ge\ldots\ge\theta_k>0$ as the non-zero eigenvalues of $I_f$ (where $k$ is the dimension of the range of $I_f$), and eigenfunctions $r_i$ corresponding to $\theta_i$. Therefore, $\{r_1,\ldots,r_k\}$ is an orthonormal set in $L^2[0,1]$, and by assumption, each $r_i$ is Riemann integrable (see Lemma \ref{lemma:RIeigenfunctions} in Section \ref{sec:app}). Also, for any $g\in L^2[0,1]$ one has
$$I_f(g) =\sum_{i=1}^k \theta_i \langle \,  r_i, g\rangle_{L^2[0,1]} r_i.$$
Note that this gives
$$ \int_0^1 \left(\sum_{i=1}^k \theta_i r_i(x) r_i(y) g(y)\right)\, dy= \int_0^1 f(x,y) g(y)\,  dy \, \text{ for almost all $x\in [0,1]$}.$$
Since $g$ is an arbitrary function in $L^2[0,1]$ this immediately gives
\[
f(x,y)=\sum_{i=1}^k\theta_ir_i(x)r_i(y)\,,\text{ for almost all }(x,y)\in[0,1]\times[0,1].
\]

Since the functions on both sides of the above equation are Riemann integrable, the corresponding Riemann sums are approximately equal, and hence there is no loss of generality in assuming that the above equality holds for every $x$ and $y$. 

That is, we \textbf{now assume that}
\begin{equation}
\label{eq.decomp}f(x,y)=\sum_{i=1}^k\theta_ir_i(x)r_i(y)\ge0\,,\text{ for all }(x,y)\in[0,1]\times[0,1]\,,
\end{equation}
where $\theta_1\ge\ldots\ge\theta_k>0$ and $\{r_1,\ldots,r_k\}$ is an orthonormal set in $L^2[0,1]$. The assumptions on $r_1,\ldots,r_k$ are listed below for easy reference.

\noindent\textbf{Assumption F1.} The functions $r_1,\ldots,r_k$ from $[0,1]$ to $\bbr$ are bounded and Riemann integrable.

\noindent\textbf{Assumption F2.} For each $i=1,\ldots,k$, $r_i$ is Lipschitz, that is, 
\[
|r_i(x)-r_i(y)|\le K_i|x-y|\,,
\]
for some fixed $K_i<\infty$. This is clearly stronger than Assumption F1, and will be needed in a few results. A consequence of this assumption is that there exists $K$ such that
\begin{equation}
\label{eq.lip}\left|f(x,y)-f(x^\prime,y^\prime)\right|\le K\left(|x-x^\prime|+|y-y^\prime|\right)\,,0\le x,x^\prime,y,y^\prime\le1\,.
\end{equation}

Let $(\vep_N:N\ge1)$ be a real sequence satisfying
\[
0<\vep_N\le\left[\sup_{0\le x,y\le1}f(x,y)\right]^{-1}\,,N\ge1\,.
\]
The following will be the bare minimum assumption for all the results.

\noindent\textbf{Assumption E1.}
For some $\xi>8$, fixed once and for all, 
\[
\lim_{N\to\infty}\frac1{N\vep_N}(\log N)^\xi=0\,,
\]
that is, \eqref{eq.vep} holds. Furthermore,
\begin{equation}
\label{eq.veplim}\lim_{N\to\infty}\vep_N=\vep_\infty\,,
\end{equation}
for some $\vep_\infty\ge0$. It's worth emphasizing that we do not assume that $\vep_N$ necessarily goes to zero, although that may be the case.

For one result, we shall have to make a stronger assumption on $(\vep_N)$ which is the following.

\noindent\textbf{Assumption E2.} As $N\to\infty$,
\begin{equation}
\label{t.flac.eq1}N^{-2/3}\ll\vep_N\ll1\,.
\end{equation}

For $N\ge1$, let $\bbg_N$ be an inhomogeneous Erd\H{o}s-R\'enyi graph where an edge is placed between vertices $i$ and $j$ with probability $\vep_N f(i/N,j/N)$, for $i\le j$, the choice being made independently for each pair in $\{(i,j):1\le i\le j\le N\}$. Note that we allow self-loops. Let $A_N$ be the adjacency matrix of $\bbg_N$. In other words, $A_N$ is an $N\times N$ symmetric matrix, where $\{A_N(i,j):1\le i\le j\le N\}$ is a collection of independent random variable, and
\[
A_N(i,j)\sim\text{Bernoulli}\left(\vep_N f\left(\frac iN,\frac jN\right)\right)\,,\,1\le i\le j\le N\,.
\]

A few more notations are needed for stating the main results. For a moment, set $\theta_0=\infty$ and $\theta_{k+1}=-\infty$, and define the set of indices $i$ for which $\theta_i$ is isolated as follows:
\[
{\mathcal I}=\{1\le i\le k: \theta_{i-1}>\theta_i>\theta_{i+1}\}\,.
\] 
For an $N\times N$ real symmetric matrix $M$, let $\lambda_1(M)\ge\ldots\ge\lambda_N(M)$ denote its eigenvalues, as mentioned in Section \ref{sec:intro}. Finally, after the following definition, the main results will be stated.

\begin{defn}
A sequence of events $E_N$ occurs \emph{with high probability}, abbreviated as w.h.p., if
\[
P(E_N^c)=O\left( e^{-(\log N)^\eta}\right)\,,
\]
for some $\eta>1$. For random variables $Y_N,Z_N$, 
\[
Y_N=O_{hp}(Z_N)\,,
\]
means there exists a deterministic finite constant $C$ such that
\[
|Y_N|\le C|Z_N|\text{ w.h.p.}\,,
\]
and
\[
Y_N=o_{hp}(Z_N)\,,
\]
means that  for all $\delta>0$,
\[
|Y_N|\le\delta|Z_N|\text{ w.h.p.}
\]

We shall say
\[
Y_N=O_p(Z_N)\,,
\]
to mean that 
\[
\lim_{x\to\infty}\sup_{N\ge1}P(|Y_N|>x|Z_N|)=0\,,
\]
and 
\[
Y_N=o_p(Z_N)\,,
\]
to mean that for all $\delta>0$,
\[
\lim_{N\to\infty}P(|Y_N|>\delta|Z_N|)=0\,.
\]
\end{defn}

The reader may note that if $Z_N\neq0$ a.s., then ``$Y_N=O_{p}(Z_N)$'' and ``$Y_N=o_{p}(Z_N)$'' are equivalent to ``$(Z_N^{-1}Y_N:N\ge1)$ is stochastically tight'' and ``$Z_N^{-1}Y_N\prob0$'', respectively. Besides, ``$Y_N=O_{hp}(Z_N)$'' is a much stronger statement than ``$Y_N=O_p(Z_N)$'', and so is ``$Y_N=o_{hp}(Z_N)$''  than ``$Y_N=o_{p}(Z_N)$''. 

In the rest of the paper, the subscript `$N$' is dropped from notations like $A_N$, $W_N$, $\vep_N$ etc.\ and the ones that will be introduced. The first result is about the first order behaviour of $\lambda_i(A)$.

\begin{theorem} 
\label{t.inprob} Under Assumptions E1.\ and F1., for every $1\le i\le k$,
\[
\lambda_i(A)=N\vep\theta_i\left(1+o_{hp}(1)\right)\,.
\]
\end{theorem}

An immediate consequence of the above is that for all $1\le i\le k$, $\lambda_i(A)$ is non-zero w.h.p.\ and hence dividing by the same is allowed, as done in the next result. Define 
\begin{equation}
\label{eq.defej}e_i=\left[
\begin{matrix}
N^{-1/2}r_i(1/N)\\
N^{-1/2}r_i(2/N)\\
\vdots\\
N^{-1/2}r_i(1)\\
\end{matrix}
\right]\,,\,1\le i\le k\,.
\end{equation}

The second main result studies the asymptotic behaviour of $\lambda_i(A)$, for $i\in\mathcal I$, after appropriate centering and scaling.

\begin{theorem}
\label{t.secorder}  Under Assumptions E1.\ and F1., for every $i\in\cal I$, as $N\to\infty$,
\[
\lambda_i(A)=\E\left(\lambda_i(A)\right)+\frac{N\theta_i\vep}{\lambda_i(A)}e_i^\prime We_i+o_p(\sqrt\vep)\,,
\]
where $W$ is as defined in \eqref{eq.defw}.
\end{theorem}

The next result is the corollary of the previous two.

\begin{theorem}\label{t.main}  Under Assumptions E1.\ and F1.,
if $\mathcal I$ is a non-empty set, then as $N\to\infty$,
\begin{equation}
\label{t.main.claim}\left(\vep^{-1/2}\left(\lambda_i(A)-\E[\lambda_i(A)]\right):i\in{\mathcal I}\right)\Rightarrow\left(G_i:i\in{\mathcal I}\right)\,,
\end{equation}
where the right hand side is a multivariate normal random vector in $\bbr^{|{\mathcal I}|}$, with mean zero and
\begin{equation}
\label{eq.cov}\Cov(G_i,G_j)=2\int_0^1\int_0^1 r_i(x)r_i(y)r_j(x)r_j(y)f(x,y)\left[1-\vep_\infty f(x,y)\right]\,dx\,dy  \,,
\end{equation}
for all $i,j\in\mathcal I$.
\end{theorem}

For $i,j\in\mathcal I$, that $W$ is a symmetric matrix whose upper triangular entries are independent and zero mean implies that as $N\to\infty$,
\begin{align}
\nonumber&Cov\left(e_i^\prime We_i,e_j^\prime We_j\right)\\
\nonumber&\sim4\sum_{1\le k\le l\le N}\Cov\left(e_i(k)W(k,l)e_i(l),e_j(k)W(k,l)e_j(l)\right)\\
\nonumber&=4N^{-2}\sum_{1\le k\le l\le N}r_i\left(\frac kN\right)r_i\left(\frac lN\right)r_j\left(\frac kN\right)r_j\left(\frac lN\right)\vep f\left(\frac kN,\frac lN\right)\left[1-\vep f\left(\frac kN,\frac lN\right)\right]\\
\nonumber&\sim4\vep\int_0^1\,dx\,\int_x^1\,dy\,r_i(x)r_i(y)r_j(x)r_j(y)f(x,y)\left[1-\vep_\infty f(x,y)\right]\\
\label{eq.cov.rhs}&=2\vep\int_0^1\int_0^1r_i(x)r_i(y)r_j(x)r_j(y)f(x,y)\left[1-\vep_\infty f(x,y)\right]\,dx\,dy\,.
\end{align}
With the help of the above, it may be checked that the Lindeberg-L\'evy central limit theorem implies that as $N\to\infty$,
\begin{align}\label{eq.clt}
\left(\vep^{-1/2}e_i^\prime We_i:i\in\cal I\right)\Rightarrow(G_i:i\in\cal I)\,,
\end{align}
where the right hand side is a zero mean Gaussian vector with covariance given by \eqref{eq.cov}.  Therefore, Theorem \ref{t.main} would follow from  Theorems \ref{t.inprob} and \ref{t.secorder}.

\begin{remark}
If $f>0$ a.e.\ on $[0,1]\times[0,1]$, then the Krein-Rutman theorem (see Lemma~\ref{lemma:KRT}) implies that $1\in\cal I$, and that $r_1>0$ a.e. Thus, in this case, if $\vep_\infty=0$, then
\[
\Var(G_1)=2\int_0^1\int_0^1r_1(x)^2r_1(y)^2f(x,y)\,dx\,dy>0\,.
\]
\end{remark}

\begin{remark}\label{rem.indep} 
For a fixed $\theta>1$, define
\[
f(x,y)=\theta\one\left(x\vee y<\frac12\right)+\one\left(x\wedge y>\frac12\right)\,,0\le x,y\le1\,.
\]
In this case, the integral operator associated with $f$ has exactly two non-zero eigenvalues, which are $\theta/2$ and $1/2$, with corresponding normalized eigenfunctions $r_1(x)=\sqrt2\one(x<1/2)$ and $r_2(x)=\sqrt2\one(x\ge1/2)$, respectively. Let $(\vep_N)$ satisfy Assumption E1.\ and suppose that $\vep_\infty=0$. Theorem \ref{t.main} implies that as $N\to\infty$,
\[
\left(\vep^{-1/2}\left(\lambda_1(A)-\E[\lambda_1(A)]\right),\vep^{-1/2}\left(\lambda_2(A)-\E[\lambda_2(A)]\right)\right)\Rightarrow\left(G_1,G_2\right)\,,
\]
where the right hand side has a bivariate normal distribution with mean zero. Furthermore, since $r_1r_2$ is identically zero, it follows that $G_1$ and $G_2$ are uncorrelated and hence independent. 
\end{remark}

\begin{remark}
That the claim of Theorem \ref{t.main} may not hold if $i\notin\cal I$ is evident from the following example. As in Remark \ref{rem.indep}, suppose that Assumption E1.\ holds and that $\vep_\infty=0$. Let
\[
f(x,y)=\one\left(x\vee y<\frac12\right)+\one\left(x\wedge y>\frac12\right)\,,0\le x,y\le1\,.
\]
In this case, the integral operator associated with $f$ has exactly one non-zero eigenvalue, which is $1/2$, and that has multiplicity $2$, with eigenfunctions $r_1$ and $r_2$ as in Remark \ref{rem.indep}. In other words, $f$ doesn't have any simple eigenvalue. 

Theorem \ref{t.main} itself implies that there exists $\beta_N\in\bbr$ such that
\[
\vep^{-1/2}\left(\lambda_1(A)-\beta\right)\Rightarrow G_1\vee G_2\,,
\]
where $G_1$ and $G_2$ are independent from normal with mean $0$. Furthermore,
\[
\Var(G_1)=2\int_0^1\int_0^1 r_1(x)^2r_1(y)^2f(x,y)\,dx\,dy=8\int_0^1\int_0^1\one\left(x\vee y\le\frac12\right)\,dx\,dy=2\,.
\]
That is, $G_1$ and $G_2$ are i.i.d.\ from $N(0,2)$. Hence, there doesn't exist a centering and a scaling by which $\lambda_1(A)$ converges weakly to a non-degenerate normal distribution.  
\end{remark}

The next main result of the paper studies asymptotics of $\E(\lambda_i(A))$ for $i\in\cal I$. 

\begin{theorem}\label{t.mean} Under Assumptions E1.\ and F2., it holds for all $i\in\cal I$,
\[
\E\left[\lambda_i(A)\right]=\lambda_i(B)+O\left(\sqrt\vep+(N\vep)^{-1}\right)\,,
\]
where $B$ is a $k\times k$ symmetric deterministic matrix, depending on $N$, defined by
\[
B(j,l)=\sqrt{\theta_j\theta_l}N\vep e_j^\prime e_l+\theta_i^{-2}\sqrt{\theta_j\theta_l}(N\vep)^{-1}\E\left(e_j^\prime W^2 e_l\right)\,,1\le j,l\le k\,,
\]
and $e_j$ and $W$ are as defined in \eqref{eq.defej} and \eqref{eq.defw}, respectively.
\end{theorem}

The next result studies the asymptotic behaviour of the normalized eigenvector corresponding to $\lambda_i(A)$, again for isolated vertices $i$. It is shown that the same is asymptotically aligned with $e_i$, and hence it is asymptotically orthogonal to $e_j$. Upper bounds on rates of convergence are obtained.

\begin{theorem}\label{t.eigvec} As in Theorem \ref{t.mean}, let Assumptions E1.\ and F2.\ hold. Then, for a fixed $i\in\cal I$, 
\begin{equation}
\label{t.eigvec.claim1}\lim_{N\to\infty}P\left(\lambda_i(A)\text{ is an eigenvalue of multiplicity }1\right)=1\,.
\end{equation}
If $v$ is an eigenvector, with $L^2$-norm $1$, of $A$ corresponding to $\lambda_i(A)$,  such that
\begin{equation}
\label{t.eigenvec.align}e_i'v\ge0\,,
\end{equation}
then
\begin{equation}
\label{t.eigvec.claim2}e_i^\prime v=1+O_p\left((N\vep)^{-1}\right)\,,
\end{equation}
that is, $N\vep(1-e_i^\prime v)$ is stochastically tight. When $k\ge2$, it holds that
\begin{equation}
\label{t.eigvec.claim3}e_j^\prime v=O_p\left((N\vep)^{-1}\right)\,,\,j\in\{1,\ldots,k\}\setminus\{i\}\,.
\end{equation}
\end{theorem}

\begin{remark}
The assumption \eqref{t.eigenvec.align} is needed in the above result to ensure that $v$ has the correct sign.  If $v$ is replaced by $-v$,  which is also a norm-one eigenvector of $A$ corresponding to $\lambda_i(A)$,  then \eqref{t.eigvec.claim2} would become
\[
e_i^\prime v=-1+O_p\left((N\vep)^{-1}\right)\,.
\]
Arguments similar to those in the proof of Lemma \ref{lemma:KRT} of the Appendix show that $r_1$ can be chosen to be non-negative.  The Perron Frobenius theorem, see \cite{ninio:1976},  implies that for $i=1$, $v$ can be chosen to have non-negative entries.  For this choice, $e_1'v\ge0$, that is,  \eqref{t.eigenvec.align} is automatically satisfied.
\end{remark}

The last main result of this paper studies finer fluctuations of  \eqref{t.eigvec.claim3} under an additional condition.

\begin{theorem}\label{t.flac}
Let $k\ge2$, $i\in\cal I$ and Assumptions E2.\ and F2.\ hold.
If $v$ is as in Theorem \ref{t.eigvec}, then, for all $j\in\{1,\ldots,k\}\setminus\{i\}$,
\[
e_j^\prime v=\frac1{\theta_i-\theta_j}\left[\theta_i\frac1{\lambda_i(A)}e_i^\prime We_j+(N\vep)^{-2}\frac1{\theta_i}\E\left(e_i^\prime W^2 e_j\right)\right]+o_p\left(\frac1{N\sqrt\vep}\right)\,.
\]
\end{theorem}

\begin{remark}\label{rem1}
An immediate consequence of Theorem \ref{t.flac} is that under Assumption E2., there exists a deterministic sequence $(z_N:N\ge1)$ given by
\[
z=\frac1{(N\vep)^2\theta_i(\theta_i-\theta_j)}\E\left(e_i^\prime W^2e_j\right)\,,
\]
such that as $N\to\infty$,
\begin{equation}
\label{eq.defZij}Z_{ij}=N\sqrt\vep\left(e_j^\prime v-z\right)
\end{equation}
converges weakly to a normal distribution with mean zero, for all $i\in\mathcal I$ and $j\in\{1,\ldots,k\}\setminus\{i\}$. Furthermore, the convergence holds jointly for all $i$ and $j$ satisfying the above. This, along with \eqref{t.main.claim}, implies that the collection
\[
\left(Z_{ij}:i\in\mathcal I, j\in\{1,\ldots,k\}\setminus\{i\}\right)\cup\left(\vep^{-1/2}\left(\lambda_i(A)-\E\left[\lambda_i(A)\right]\right):i\in\mathcal I\right)
\]
converges weakly, as $N\to\infty$, to 
\[
\left(G_{ij}:i\in\mathcal I, j\in\{1,\ldots,k\}\setminus\{i\}\right)\cup\left(G_i:i\in\mathcal I\right)
\]
which is a zero mean Gaussian vector in $\bbr^{k|{\cal I}|}$. The covariance matrix of $(G_i)$ is as in \eqref{eq.cov}, and $\Cov(G_{ij},G_{i^\prime\,j^\prime})$ and $\Cov(G_{ij},G_{i^\prime})$ are not hard to calculate by proceeding as in \eqref{eq.cov.rhs}. 
\end{remark}

\section{Examples and special cases}\label{sec:eg}

\subsection*{The rank one case}
Let us consider the special case of $k=1$, that is,
\[
f(x,y)=\theta r(x)r(y)\,,
\]
for some $\theta>0$, and a bounded Riemann integrable $r:[0,1]\to[0,\infty)$ satisfying
\[
\int_0^1r(x)^2\,dx=1\,.
\]
In this case, Theorem \ref{t.main} implies that 
\[
\vep^{-1/2}\left(\lambda_1(A)-\E\left(\lambda_1(A)\right)\right)\Rightarrow G_1\,,
\]
as $N\to\infty$, where
\[
G_1\sim N\left(0\,,\sigma^2\right)\,,
\]
with
\[
\sigma^2=2\theta\left(\int_0^1r(x)^3\,dx\right)^2-2\theta^2\vep_\infty\left(\int_0^1r(x)^4\,dx\right)^2\,.
\]
If $r$ is Lipschitz and $\vep_\infty=0$, then the claim of Theorem \ref{t.mean} boils down to
\begin{equation}
\label{eg.eq1}\E\left[\lambda_1(A)\right]=\theta N\vep e_1^\prime e_1+(N\vep\theta)^{-1}\E\left(e_1^\prime W^2 e_1\right)+O\left(\sqrt\vep+(N\vep)^{-1}\right)\,,
\end{equation}
where
\[
e_1=N^{-1/2}\left[r(1/N)\,r(2/N)\,\ldots r(1)\right]^\prime\,.
\]
Lipschitz continuity of $r$ implies that
\[
e_1^\prime e_1=1+O\left(N^{-1}\right)\,,
\]
and hence \eqref{eg.eq1} becomes
\begin{equation}
\label{eg.eq2}
\E\left[\lambda_1(A)\right]=\theta N\vep+(N\vep\theta)^{-1}\E\left(e_1^\prime W^2 e_1\right)+O\left(\sqrt\vep+(N\vep)^{-1}\right)\,.
\end{equation}
Clearly,
\begin{align*}
&\E\left(e_1^\prime W^2 e_1\right)\\
=&\frac1N\sum_{i=1}^Nr\left(\frac iN\right)^2\E\left[W^2(i,i)\right]\\
=&\frac1N\sum_{i=1}^Nr\left(\frac iN\right)^2\sum_{1\le j\le N,\,j\neq i}\vep f\left(\frac iN,\frac jN\right)\left(1-\vep f\left(\frac iN,\frac jN\right)\right)\\
=&\theta\vep N^{-1}\sum_{1\le i\neq j\le N}r\left(\frac iN\right)^3r\left(\frac jN\right)+O\left(N^{-1}\vep^2\right)\\
=&N\theta\vep\int_0^1r(x)^3\,dx\int_0^1r(y)\,dy+O(\vep)\,.
\end{align*}
In conjunction with \eqref{eg.eq2} this yields
\[
\E\left[\lambda_1(A)\right]=\theta N\vep+\int_0^1r(x)^3\,dx\int_0^1r(y)\,dy+O\left(\sqrt\vep+(N\vep)^{-1}\right)\,.
\]

\subsection*{Stochastic block model}
Another important example is the stochastic block model, defined as follows. Suppose that
\[
f(x,y)=\sum_{i,j=1}^kp(i,j)\one_{B_i}(x)\one_{B_j}(y)\,,0\le x,y\le1\,,
\]
where $p$ is a $k\times k$ symmetric positive definite matrix, and $B_1,\ldots,B_k$ are disjoint Borel subsets of $[0,1]$ whose boundaries are sets of measure zero, that is, their indicators are Riemann integrable. We show below how to compute the eigenvalues and eigenfunctions of $I_f$, the integral operator associated with $f$.

Let $\beta_i$ denote the Lebesgue measure of $B_i$, which we assume without loss of generality to be strictly positive. Rewrite
\[
f(x,y)=\sum_{i,j=1}^k\tilde p(i,j)s_i(x)s_j(y)\,,
\]
where 
\[
\tilde p(i,j)=p(i,j)\sqrt{\beta_i\beta_j}\,,1\le i,j\le k\,,
\]
and
\[
s_i=\beta_i^{-1/2}\one_{B_i}\,,1\le i\le k\,.
\]
Thus, $\{s_1,\ldots,s_k\}$ is an orthonormal set in $L^2[0,1]$. Let
\[
\tilde p=U^\prime DU\,,
\]
be a spectral decomposition of $\tilde p$, where $U$ is a $k\times k$ orthogonal matrix, and
\[
D=\diag(\theta_1,\ldots,\theta_k)\,,
\]
for some $\theta_1\ge\ldots\ge\theta_k>0$. 

Define functions $r_1,\ldots,r_k$ by
\[
\left[
\begin{matrix}
r_1(x)\\
\vdots\\
r_k(x)
\end{matrix}
\right]
=U
\left[
\begin{matrix}
s_1(x)\\
\vdots\\
s_k(x)
\end{matrix}
\right],\,x\in[0,1]\,.
\]
It is easy to see that $r_1,\ldots,r_k$ are orthonormal in $L^2[0,1]$, and for $0\le x,y\le1,$
\begin{align*}
f(x,y)=&\left[s_1(x)\,\ldots\,s_k(x)\right]\tilde p\left[s_1(x)\,\ldots\,s_k(x)\right]^\prime\\
=&\left[r_1(x)\,\ldots\,r_k(x)\right]U\tilde p U^\prime\left[r_1(x)\,\ldots\,r_k(x)\right]^\prime\\
=&\sum_{i=1}^k\theta_ir_i(x)r_i(y)\,.
\end{align*}
Thus, $\theta_1,\ldots,\theta_k$ are the eigenvalues of $I_f$, and $r_1,\ldots,r_k$ are the corresponding eigenfunctions.

\section{Estimates}\label{sec:est}
In this section, we'll record a few estimates that will subsequently be used in the proof. Since their proofs are routine, they are being postponed to Section \ref{sec:app} which is the Appendix. Let $W$ be as defined in \eqref{eq.defw}. Throughout this section, Assumptions E1.\ and F1.\ will be in force. 

\begin{lemma}\label{est.l1}
There exist constants $C_1$, $C_2 > 0$ such that 
\begin{equation}\label{eq:est.l1.1}
    P \left( \| W \| \ge 2 \sqrt{MN\vep} + C_1 (N \vep )^{1/4} ( \log N)^{\xi/4}\right) \le e^{-C_2(\log N)^{\xi/4}},
\end{equation}
where $M=\sup_{0\le x,y\le 1}f(x,y)$. Consequently, 
\[
\|W\|=O_{hp}\left(\sqrt{N\vep}\right)\,.
\]
\end{lemma} 

The notations $e_1$ and $e_2$ introduced in the next lemma and used in the subsequent lemmas should not be confused with $e_j$ defined in \eqref{eq.defej}. Continuing to suppress `$N$' in the subscript, let 
\begin{align*}
L=&[\log N]\,,
\end{align*}
where $[x]$ is the largest integer less than or equal to $x$. 

\begin{lemma}\label{est.l2}
There exists $0<C_1<\infty$ such that if $e_1$ and $e_2$ are $N\times1$ vectors with each entry in $[-1/\sqrt N,1/\sqrt N]$, then
\[
\left|\E\left(e_1^\prime W^ne_2\right)\right|\le (C_1N\vep)^{n/2}\,,2\le n\le L\,.
\]
\end{lemma}

\begin{lemma}\label{est.l3}
There exists $\eta_1>1$ such that for $e_1,e_2$ as in Lemma \ref{est.l2}, it holds that
\begin{align}
&\max_{2\le n\le L}P\left(\left|e_1^\prime W^ne_2-\E\left(e_1^\prime W^ne_2\right)\right|>N^{(n-1)/2}\vep^{n/2}(\log N)^{n\xi/4}\right)\nonumber\\
=&O\left( e^{-(\log N)^{\eta_1}}\right)\,,\label{est.l3.1}
\end{align}
where $\xi$ is as in \eqref{eq.vep}. In addition,
\begin{equation}\label{est.l3.2}
e_1^\prime We_2=o_{hp}\left(N\vep\right)\,.
\end{equation}
\end{lemma}

\begin{lemma}\label{est.l4}
If $e_1,e_2$ are as in Lemma \ref{est.l2}, then
\begin{equation}
\label{est.l4.eq1}\Var\left(e_1^\prime We_2\right)=O(\vep)\,,
\end{equation}
and
\begin{equation}
\label{est.l4.eq2}\E\left(e_1^\prime W^3e_2\right)=O(N\vep)\,.
\end{equation}
\end{lemma}

\section{Proof of the main results}\label{sec:proof} 
This section is devoted to the proof of the main results. The section is split into several subsections, each containing the proof of one main result, for the ease of reading. Unless mentioned otherwise, Assumptions E1.\ and F1.\ are made.

\subsection{Proof of Theorem \ref{t.inprob}}
We start with showing that Theorem \ref{t.inprob} is a corollary of Lemma \ref{est.l1}. At this point, it should be clarified that throughout this section, $e_j$ will always be as defined in \eqref{eq.defej}.

\begin{proof}[Proof of Theorem \ref{t.inprob}]
For a fixed $i\in\{1,\ldots,k\}$, it follows that
\begin{align*}
\left|\lambda_i(A)-\lambda_i\left(\E(A)\right)\right|\le\|W\|=O_{hp}\left((N\vep)^{1/2}\right)\,,
\end{align*}
by Lemma \ref{est.l1}. In order to complete the proof, it suffices to show that
\[
\lim_{N\to\infty}(N\vep)^{-1}\lambda_i\left(\E(A)\right)=\theta_i\,,
\]
which however follows from the observation that \eqref{eq.decomp} implies that
\begin{equation}
\label{t.inprob.eq1}\E(A)=N\vep\sum_{j=1}^k\theta_je_je_j^\prime\,.
\end{equation}
This completes the proof. 
\end{proof}

\subsection{Proof of Theorem \ref{t.secorder}}
Proceeding towards the proof of Theorem \ref{t.secorder}, let us fix $i\in\cal I$, once and for all, denote
\[
\mu=\lambda_i(A)\,,
\]
and let $V$ be a $k\times k$ real symmetric matrix, depending on $N$ which is suppressed in the notation, defined by
\[
V(j,l)=
\begin{cases}
N\vep\sqrt{\theta_j\theta_l}\,e_j^\prime\left(I-\frac1\mu W\right)^{-1}e_l,&\text{if }\|W\|<\mu\,,\\
0,&\text{else}\,,
\end{cases}
\]
for all $1\le j,l\le k$. It should be noted that if $\|W\|<\mu$, then $I-W/\mu$ is invertible. 

The proof of Theorem \ref{t.secorder}, which is the main content of this paper, involves several steps, and hence it is imperative to sketch the outline of the proof for the reader's convenience, which is done below.
\begin{enumerate}
\item The first major step of the proof is to show that w.h.p., $\mu$ exactly equals $\lambda_i(V)$. This is done in Lemma \ref{proof.l1}. When $i=k=1$, the matrix $V$ is a scalar and hence in that case the equation boils down to $\mu=V$ which is a consequence of the resolvent equation. For higher values of $k$, the Gershgorin circle theorem is employed for the desired claim. Therefore, this step is a novelty of the given proof.
\item The next step in the proof is to write $\mu$ as the solution of an equation of the form
\[
\mu=\lambda_i\left(\sum_{n=0}^L\mu^{-n}Y_n\right)+Error\,,
\]
for suitable matrices $Y_1,Y_2,\ldots$. This is done in Lemma \ref{proof.l2}.
\item The third step is to replace $Y_n$ by $\E(Y_n)$ in the equation obtained in the above step, for $n\ge2$. This is done in Lemma \ref{proof.l3}.
\item Arguably the most important step in the proof is to obtain an equation of the form
\[
\mu=\bar\mu+\mu^{-1}\zeta+Error\,,
\]
for some deterministic $\bar\mu$ depending on $N$ and random $\zeta$. Once again, this is achieved from the previous step with the help of the Gershgorin circle theorem and other linear algebraic tools. This is done in Lemma \ref{proof.l6}. 
\item The final step of the proof is to show that $\bar\mu$ of the above step can be replaced by $\E(\mu)$. 
\end{enumerate}

Now let us proceed towards executing the above steps for proving Theorem \ref{t.secorder}. As the zeroth step, we show that $V/N\vep$ converges to $\diag(\theta_1,\ldots,\theta_k)$, that is, the $k\times k$ diagonal matrix with diagonal entries $\theta_1,\ldots,\theta_k$, w.h.p.

\begin{lemma}\label{proof.l0}
As $N\to\infty$,
\[
V(j,l)=N\vep\theta_j\left(\one(j=l)+o_{hp}(1)\right)\,,1\le j,l\le k\,.
\]
\end{lemma}

\begin{proof}
For fixed $1\le j,l\le k$, writing
\[
\left(I-\frac1\mu W\right)^{-1}=I+O_{hp}\left(\mu^{-1}\|W\|\right)\,,
\]
we get that
\[
V(j,l)=N\vep\sqrt{\theta_j\theta_l}\left(e_j^\prime e_l+\frac1\mu O_{hp}(\|W\|)\right)\,.
\]
Since
\begin{equation}
\label{proof.l0.eq1}\lim_{N\to\infty}e_j^\prime e_l=\one(j=l)\,,
\end{equation}
and
\[
\|W\|=o_{hp}(\mu)
\]
by Lemma \ref{est.l1} and Theorem \ref{t.inprob}, the proof follows.
\end{proof}

The next step, which is one of the main steps in the proof of Theorem \ref{t.secorder}, shows that the $i$-th eigenvalues of $A$ and $V$ are exactly equal w.h.p.

\begin{lemma}\label{proof.l1}
With high probability, 
\[
\mu=\lambda_i(V)\,.
\]
\end{lemma}

The proof of the above lemma is based on the following fact which is a direct consequence of the Gershgorin circle theorem; see Theorem 1.6, pg 8 of \cite{Varga:book}. 

\begin{fact}\label{f.gct} Suppose that $U$ is an $n\times n$ real symmetric matrix. Define
\[
R_l=\sum_{1\le j\le n,\,j\neq l}|U(j,l)|\,,\,1\le l\le n\,.
\]
If for some $1\le m\le n$ it holds that
\begin{equation}
\label{f.gct.eq1}U(m,m)+R_m<U(l,l)-R_l\,,\text{ for all }1\le l\le m-1\,,
\end{equation}
and
\begin{equation}
\label{f.gct.eq2}U(m,m)-R_m>U(l,l)+R_l\,,\text{ for all }m+1\le l\le n\,,
\end{equation}
then 
\[
\bigl\{\lambda_1(U),\ldots,\lambda_n(U)\bigr\}\setminus\left(\bigcup_{1\le l\le k,\,l\neq m}\left[U(l,l)-R_l,U(l,l)+R_l\right]\right)=\bigl\{\lambda_m(U)\bigr\}\,.
\]
\end{fact}

\begin{remark}
The assumptions \eqref{f.gct.eq1} and \eqref{f.gct.eq2} of Fact \ref{f.gct} mean that the Gershgorin disk containing the $m$-th largest eigenvalue is disjoint from any other Gershgorin disk. 
\end{remark}

\begin{proof}[Proof of Lemma \ref{proof.l1}]
The first step is to show that 
\begin{equation}
\label{proof.l1.eq1}\mu\in\bigl\{\lambda_1(V),\ldots,\lambda_k(V)\bigr\}\text{ w.h.p.}
\end{equation}
To that end, fix $N\ge1$ and a sample point for which $\|W\|<\mu$. The following calculations are done for that fixed sample point. 

Let $v$ be an eigenvector of $A$, with norm $1$, corresponding to $\lambda_i(A)$. That is,
\begin{equation}
\label{proof.l1.eq2}
\mu v=Av=Wv+N\vep\sum_{l=1}^k \theta_l(e_l^\prime v)e_l\,,
\end{equation}
by \eqref{t.inprob.eq1}. Since $\mu>\|W\|$, $\mu I-W$ is invertible, and hence
\begin{equation}
\label{eq.master}v=N\vep\sum_{l=1}^k\theta_l(e_l^\prime v)\left(\mu I-W\right)^{-1}e_l\,.
\end{equation}
Fixing $j\in\{1,\ldots,k\}$ and premultiplying the above by $\sqrt{\theta_j}\mu e_j^\prime$ yields
\[
\mu\sqrt{\theta_j}(e_j^\prime v)=N\vep\sum_{l=1}^k\sqrt{\theta_j}\theta_l(e_l^\prime v)e_j^\prime\left(I-\frac1\mu W\right)^{-1}e_l=\sum_{l=1}^kV(j,l)\sqrt{\theta_l}(e_l^\prime v)\,.
\]
As the above holds for all $1\le j\le k$, this means that if
\begin{equation}
\label{proof.l1.eqnew1}u=\left[\sqrt{\theta_1}(e_1^\prime v)\ldots\sqrt{\theta_k}(e_k^\prime v)\right]^\prime\,,
\end{equation}
then
\begin{equation}
\label{proof.l1.eqnew2}Vu=\mu u\,.
\end{equation}
Recalling that in the above calculations a sample point is fixed such that $\|W\|<\mu$, what we have shown, in other words, is that a vector $u$ satisfying the above exists w.h.p.

In order to complete the proof of \eqref{proof.l1.eq1}, it suffices to show that $u$ is a non-null vector w.h.p. To that end, premultiply \eqref{proof.l1.eq2} by $v^\prime$ to obtain that
\[
\mu=v^\prime Wv+N\vep\|u\|^2\,.
\]
Dividing both sides by $N\vep$ and using Lemma \ref{est.l1} implies that
\[
\|u\|^2=\theta_i+o_{hp}(1)\,.
\]
Thus, $u$ is a non-null vector w.h.p. From this and \eqref{proof.l1.eqnew2}, \eqref{proof.l1.eq1} follows. 

Lemma \ref{proof.l0} shows that for all $l\in\{1,\ldots,i-1\}$,
\begin{align*}
&\left[V(i,i)+\sum_{1\le j\le k,\,j\neq i}|V(i,j)|\right]-\left[V(l,l)-\sum_{1\le j\le k,\,j\neq l}|V(l,j)|\right]\\
=&N\vep\left(\theta_i-\theta_l\right)(1+o_{hp}(1))\,,
\end{align*}
as $N\to\infty$. Since $i\in\cal I$, $\theta_i-\theta_l<0$, and hence
\[
V(i,i)+\sum_{1\le j\le k,\,j\neq i}|V(i,j)|<V(l,l)-\sum_{1\le j\le k,\,j\neq l}|V(l,j)|\text{ w.h.p.}
\]
A similar calculation shows that for $l\in\{i+1,\ldots,k\}$,
\[
V(i,i)-\sum_{1\le j\le k,\,j\neq i}|V(i,j)|>V(l,l)+\sum_{1\le j\le k,\,j\neq l}|V(l,j)|\text{ w.h.p.}
\]
In view of \eqref{proof.l1.eq1} and Fact \ref{f.gct}, the proof would follow once it can be shown that for all $l\in\{1,\ldots,k\}\setminus\{i\}$,
\[
\left|\mu-V(l,l)\right|>\sum_{1\le j\le k,\,j\neq l}|V(l,j)|\text{ w.h.p.}
\]
This follows, once again, by dividing both sides by $N\vep$ and using Theorem \ref{t.inprob} and Lemma \ref{proof.l0}. This completes the proof.
\end{proof}

The next step is to write
\begin{equation}
\label{eq.infty}\left(I-\frac1\mu W\right)^{-1}=\sum_{n=0}^\infty\mu^{-n}W^n\,,
\end{equation}
which is possible because $\|W\|<\mu$. Denote
\[
Z_{j,l,n}=e_j^\prime W^n e_l\,,1\le j,l\le k\,,n\ge0\,,
\]
which should not be confused with $Z_{ij}$ defined in \eqref{eq.defZij}, and for $n\ge0$, let $Y_n$ be a $k\times k$ matrix with
\[
Y_n(j,l)=\sqrt{\theta_j\theta_l}N\vep Z_{j,l,n}\,,1\le j,l\le k\,.
\]
The following bounds will be used several times.

\begin{lemma}\label{l.y1}
It holds that
\[
\E\left(\|Y_1\|\right)=O\left(N\vep^{3/2}\right)\,,
\]
and
\[
\|Y_1\|=o_{hp}\left((N\vep)^2\right)\,.
\]
\end{lemma}

\begin{proof}
Lemma \ref{est.l4} implies that
\[
\Var\left(Z_{j,l,1}\right)=O(\vep)\,,1\le j,l\le k\,.
\]
Hence,
\begin{align*}
\E\|Y_1\|=&O\left(N\vep\sum_{j,l=1}^k\E|Z_{j,l,1}|\right)\\
=&O\left(N\vep\sum_{j,l=1}^k\sqrt{\Var(Z_{j,l,1})}\right)\\
=&O\left(N\vep^{3/2}\right)\,,
\end{align*}
the equality in the second line using the fact that $Z_{j,l,1}$ has mean $0$. This proves the first claim. The second claim follows from \eqref{est.l3.2} of Lemma \ref{est.l3}. 
\end{proof}

The next step is to truncate the infinite sum in \eqref{eq.infty} to level $L$, where $L=[\log N]$ as defined before. 

\begin{lemma}\label{proof.l2}
It holds that
\[
\mu=\lambda_i\left(\sum_{n=0}^L\mu^{-n}Y_n\right)+o_{hp}\left(\sqrt\vep\right)\,.
\]
\end{lemma}

\begin{proof}
From the definition of $V$, it is immediate that for $1\le j,l\le k$,
\[
V(j,l)=N\vep\sqrt{\theta_j\theta_l}\sum_{n=0}^\infty\mu^{-n}e_j^\prime W^ne_l\,\one(\|W\|<\mu)\,,
\]
and hence
\[
V=\one(\|W\|<\mu)\sum_{n=0}^\infty\mu^{-n}Y_n\,.
\]
For the sake of notational simplicity, let us suppress $\one(\|W\|<\mu)$. Therefore, with the implicit understanding that the sum is set as zero if $\|W\|\ge\mu$, for the proof it suffices to check that
\begin{equation}
\label{proof.l2.eq1}\left\|\sum_{n=L+1}^\infty\mu^{-n}Y_n\right\|=o_{hp}(\sqrt\vep)\,.
\end{equation}
To that end, Theorem \ref{t.inprob} and Lemma \ref{est.l1} imply that
\begin{align*}
\left\|\sum_{n=L+1}^\infty\mu^{-n}Y_n\right\|\le&\sum_{n=L+1}^\infty|\mu|^{-n}\|Y_n\|\\
=&O_{hp}\left((N\vep)^{-(L-1)/2}\right)\,.
\end{align*}
In order to prove \eqref{proof.l2.eq1}, it suffices to show that as $N\to\infty$,
\begin{equation}
\label{proof.l2.eq0}-\log\vep=o\left((L-1)\log(N\vep)\right)\,.
\end{equation}
To that end, recall \eqref{eq.vep} to argue that
\begin{equation}
\label{proof.l2.eq2}N^{-1}=o(\vep)
\end{equation}
and
\begin{equation}
\label{proof.l2.eq3}\log\log N=O(\log(N\vep))\,.
\end{equation}
By \eqref{proof.l2.eq2}, it follows that
\begin{align*}
-\log\vep=&O\left(\log N\right)\\
=&o\left(\log N\log\log N\right)\\
=&o\left((L-1)\log(N\vep)\right)\,,
\end{align*}
the last line using \eqref{proof.l2.eq3}. Therefore, \eqref{proof.l2.eq0} follows, which ensures \eqref{proof.l2.eq1}, which in turn completes the proof.
\end{proof}

In the next step, $Y_n$ is replaced by its expectation for $n\ge2$.

\begin{lemma}\label{proof.l3}
It holds that
\[
\mu=\lambda_i\left(Y_0+\mu^{-1}Y_1+\sum_{n=2}^L\mu^{-n}\E(Y_n)\right)+o_{hp}\left(\sqrt\vep\right)\,.
\]
\end{lemma}

\begin{proof}
In view of Theorem \ref{t.inprob} and Lemma \ref{proof.l2}, all that has to be checked is
\begin{equation}
\label{proof.l3.eq1}\sum_{n=2}^L(N\vep)^{-n}\|Y_n-\E(Y_n)\|=o_{hp}(\sqrt\vep)\,.
\end{equation}
For that, invoke Lemma \ref{est.l3} to claim that
\[
\max_{2\le n\le L,\,1\le j,l\le k}P\left(\left|Z_{j,l,n}-\E(Z_{j,l,n})\right|>N^{(n-1)/2}\vep^{n/2}(\log N)^{n\xi/4}\right)
\]
\begin{equation}
\label{proof.l3.eq2}=O\left(e^{-(\log N)^{\eta_1}}\right)\,,
\end{equation}
where $\xi$ is as in \eqref{eq.vep}. 

Our next claim is that there exists $C_2>0$ such that for $N$ large,
\begin{equation}\label{proof.l3.eq3}
\bigcap_{2\le n\le L,1\le j,l\le k}\left[\left|Z_{j,l,n}-\E(Z_{j,l,n})\right|\le N^{(n-1)/2}\vep^{n/2}(\log N)^{n\xi/4}\right]
\end{equation}
\[
\subset\left[\sum_{n=2}^L(N\vep)^{-n}\|Y_n-\E(Y_n)\|\le C_2\sqrt\vep\left((N\vep)^{-1}(\log N)^\xi\right)^{1/2}\right]\,.
\]
To see this, suppose that the event on the left hand side holds. Then, for fixed $1\le j,l\le k$, and large $N$,
\begin{align*}
&\sum_{n=2}^L(N\vep)^{-n}\left\|Y_n(j,l)-\E\left[Y_n(j,l)\right]\right\|\\
\le&\theta_1 N\vep\sum_{n=2}^L(N\vep)^{-n}\left|Z_{j,l,n}-\E\left(Z_{j,l,n}\right)\right|\\
\le&\theta_1\sum_{n=2}^\infty(N\vep)^{-(n-1)}N^{(n-1)/2}\vep^{n/2}(\log N)^{n\xi/4}\\
=&\left[1-(N\vep)^{-1/2}(\log N)^{\xi/4}\right]^{-1}\theta_1\sqrt\vep(N\vep)^{-1/2}(\log N)^{\xi/2}\,.
\end{align*}
Thus, \eqref{proof.l3.eq3} holds for some $C_2>0$. 

Combining \eqref{proof.l3.eq2} and \eqref{proof.l3.eq3}, it follows that 
\begin{align*}
&P\left(\sum_{n=2}^L(N\vep)^{-n}\|Y_n-\E(Y_n)\|>C_2\sqrt\vep\left((N\vep)^{-1}(\log N)^\xi\right)^{1/2}\right)\\
=&O\left(\log N e^{-(\log N)^{\eta_1}}\right)\\
=&o\left(e^{-(\log N)^{(1+\eta_1)/2}}\right)\,.
\end{align*}
This, with the help of \eqref{eq.vep}, establishes \eqref{proof.l3.eq1} from which the proof follows.
\end{proof}

The goal of the next two lemmas is replacing $\mu$ by a deterministic quantity in 
\[
\sum_{n=2}^L\mu^{-n}\E(Y_n)\,.
\]

\begin{lemma}\label{proof.l4}
For $N$ large, the deterministic equation
\begin{equation}
\label{proof.l4.eq1}x=\lambda_i\left(\sum_{n=0}^Lx^{-n}\E(Y_n)\right),\,x>0\,,
\end{equation}
has a solution $\tilde\mu$ such that
\begin{equation}\label{proof.l4.eq2}
0<\liminf_{N\to\infty}(N\vep)^{-1}\tilde\mu\le\limsup_{N\to\infty}(N\vep)^{-1}\tilde\mu<\infty\,.
\end{equation}
\end{lemma}

\begin{proof}
Define a function
\[
h:(0,\infty)\to\bbr\,,
\]
by
\[
h(x)=\lambda_i\left(\sum_{n=0}^Lx^{-n}\E(Y_n)\right)\,.
\]
Our first claim is that for any fixed $x>0$,
\begin{equation}
\label{proof.l4.eq3}\lim_{N\to\infty}(N\vep)^{-1}h\left(xN\vep\right)=\theta_i\,.
\end{equation}
To that end, observe that since $\E(Y_1)=0$,
\[
h\left(xN\vep\right)=\lambda_i\left(\E(Y_0)+\sum_{n=2}^L(xN\vep)^{-n}\E(Y_n)\right)\,.
\]
Recalling that
\[
Y_0(j,l)=N\vep\sqrt{\theta_j\theta_l}\,e_j^\prime e_l\,,1\le j,l\le k\,,
\]
it follows by \eqref{proof.l0.eq1} that
\begin{equation}
\label{proof.l4.eq4}\lim_{N\to\infty}(N\vep)^{-1}\E(Y_0)=\diag(\theta_1,\ldots,\theta_k)\,.
\end{equation}
Lemma \ref{est.l2} implies that 
\[
\E(Z_{j,l,n})\le\left(O(N\vep)\right)^{n/2}\,,
\]
uniformly for $2\le n\le L$, and hence there exists $0<C_3<\infty$ with
\begin{equation}
\label{proof.l4.eq5}\|\E(Y_n)\|\le(C_3N\vep)^{n/2+1}\,,\,2\le n\le L\,.
\end{equation}
Therefore,
\[
\left\|\sum_{n=2}^L(xN\vep)^{-n}\E(Y_n)\right\|\le\sum_{n=2}^\infty(xN\vep)^{-n}(C_3N\vep)^{n/2+1}\to C_3^2x^{-2}\,,
\]
as $N\to\infty$. With the help of \eqref{proof.l4.eq4}, this implies that
\[
\lim_{N\to\infty}(N\vep)^{-1}\left(\sum_{n=0}^L(xN\vep)^{-n}\E(Y_n)\right)=\diag(\theta_1,\ldots,\theta_k)\,,
\]
and hence \eqref{proof.l4.eq3} follows. It follows that for a fixed $0<\delta<\theta_i$. 
\[
\lim_{N\to\infty}(N\vep)^{-1}\left[N\vep(\theta_i+\delta)-h\left((\theta_i+\delta)N\vep\right)\right]=\delta\,,
\]
and thus, for large $N$,
\[
N\vep(\theta_i+\delta)>h\left((\theta_i+\delta)N\vep\right)\,.
\]
Similarly, again for large $N$,
\[
N\vep(\theta_i-\delta)<h\left((\theta_i-\delta)N\vep\right)\,.
\]
Hence, for $N$ large, \eqref{proof.l4.eq1} has a solution $\tilde\mu$ in $[(N\vep)(\theta_i-\delta),(N\vep)(\theta_i+\delta)]$, which trivially satisfies \eqref{proof.l4.eq2}. Hence the proof.
\end{proof}

\begin{lemma}\label{proof.l5}
If $\tilde\mu$ is as in Lemma \ref{proof.l4}, then
\[
\mu-\tilde\mu=O_{hp}\left((N\vep)^{-1}\|Y_1\|+\sqrt\vep\right)\,.
\]
\end{lemma}

\begin{proof}
Lemmas \ref{proof.l3} and \ref{proof.l4} imply that
\begin{align*}
&|\mu-\tilde\mu|\\
=&\left|\lambda_i\left(Y_0+\mu^{-1}Y_1+\sum_{n=2}^L\mu^{-n}\E(Y_n)\right)-\lambda_i\left(\sum_{n=0}^L\tilde\mu^{-n}\E(Y_n)\right)\right|+o_{hp}(\sqrt\vep)\\
\le&\|\mu^{-1}Y_1\|+|\mu-\tilde\mu|\sum_{n=2}^L\mu^{-n}\tilde\mu^{-n}\|E(Y_n)\|\sum_{j=0}^{n-1}\mu^j\tilde\mu^{n-1-j}+o_{hp}(\sqrt\vep)\\
=&|\mu-\tilde\mu|\sum_{n=2}^L\mu^{-n}\tilde\mu^{-n}\|E(Y_n)\|\sum_{j=0}^{n-1}\mu^j\tilde\mu^{n-1-j}+O_{hp}\left((N\vep)^{-1}\|Y_1\|+\sqrt\vep\right)\,.
\end{align*}
Thus,
\begin{equation}
\label{proof.l5.eq1}|\mu-\tilde\mu|\left[1-\sum_{n=2}^L\mu^{-n}\tilde\mu^{-n}\|E(Y_n)\|\sum_{j=0}^{n-1}\mu^j\tilde\mu^{n-1-j}\right]\le O_{hp}\left((N\vep)^{-1}\|Y_1\|+\sqrt\vep\right)\,.
\end{equation}
Equations \eqref{proof.l4.eq2} and \eqref{proof.l4.eq5} imply that 
\begin{align}
\nonumber\left|\sum_{n=2}^L\mu^{-n}\tilde\mu^{-n}\|E(Y_n)\|\sum_{j=0}^{n-1}\mu^j\tilde\mu^{n-1-j}\right|=&O_{hp}\left(\sum_{n=2}^\infty n(N\vep)^{-(n+1)}(C_3N\vep)^{n/2+1}\right)\\
\nonumber=&O_{hp}\left((N\vep)^{-1}\right)\\
\label{proof.l5.eq2}=&o_{hp}(1),\,N\to\infty\,.
\end{align}
This completes the proof with the help of \eqref{proof.l5.eq1}.
\end{proof}

The next lemma is arguably the most important step in the proof of Theorem \ref{t.secorder}, the other major step being Lemma \ref{proof.l1}.

\begin{lemma}\label{proof.l6}
There exists a deterministic $\bar\mu$, which depends on $N$, such that
\[
\mu=\bar\mu+\mu^{-1}Y_1(i,i)+o_{hp}\left((N\vep)^{-1}\|Y_1\|+\sqrt\vep\right)\,.
\]
\end{lemma}

\begin{proof}
Define a $k\times k$ deterministic matrix
\[
X=\sum_{n=0}^L\tilde\mu^{-n}\E(Y_n)\,,
\]
which, as usual, depends on $N$. Lemma \ref{proof.l5} and \eqref{proof.l5.eq2} imply that
\begin{align*}
\left\|X-\sum_{n=0}^L\mu^{-n}\E(Y_n)\right\|\le&|\mu-\tilde\mu|\sum_{n=2}^L\mu^{-n}\tilde\mu^{-n}\|\E(Y_n)\|\sum_{j=0}^{n-1}\mu^j\tilde\mu^{n-1-j}\\
=&o_{hp}\left(|\mu-\tilde\mu|\right)\\
=&o_{hp}\left((N\vep)^{-1}\|Y_1\|+\sqrt\vep\right)\,.
\end{align*}
By Lemma \ref{proof.l3} it follows that
\begin{equation}
\label{proof.l6.eq1}\mu=\lambda_i\left(\mu^{-1}Y_1+X\right)+o_{hp}\left((N\vep)^{-1}\|Y_1\|+\sqrt\vep\right)\,.
\end{equation}
Let
\[
H=X+\mu^{-1}Y_1-\left(X(i,i)+\mu^{-1}Y_1(i,i)\right)I\,,
\]
\[
M=X-X(i,i)I\,,
\]
and
\[
\bar\mu=\lambda_i(X)\,.
\]
Clearly,
\begin{align*}
\lambda_i\left(\mu^{-1}Y_1+X\right)=&X(i,i)+\mu^{-1}Y_1(i,i)+\lambda_i(H)\\
=&\bar\mu-\lambda_i(M)+\mu^{-1}Y_1(i,i)+\lambda_i(H)\,.
\end{align*}
Thus, the proof would follow with the aid of \eqref{proof.l6.eq1} if it can be shown that
\begin{equation}
\label{proof.l6.eq2}\lambda_i(H)-\lambda_i(M)=o_{hp}\left((N\vep)^{-1}\|Y_1\|\right)\,.
\end{equation}
If $k=1$, then $i=1$ and hence $H=M=0$. Thus, the above is a tautology in that case. Therefore, assume without loss of generality that $k\ge2$.

Proceeding towards proving \eqref{proof.l6.eq2} when $k\ge2$, set
\begin{equation}
\label{proof.l6.u1}U_1=(N\vep)^{-1}M\,,
\end{equation}
and
\begin{equation}
\label{proof.l6.u2}U_2=(N\vep)^{-1}H\,.
\end{equation}

The main idea in the proof of \eqref{proof.l6.eq2} is to observe that the eigenvector of $U_1$ corresponding to $\lambda_i(U_1)$ is same as that of $M$ corresponding to $\lambda_i(M)$, and likewise for $U_2$ and $X$. Hence, the first step is to use this to get a bound on the differences between the eigenvectors in terms of $\|U_1-U_2\|$.

An important observation that will be used later is that
\begin{equation}
\label{proof.l6.corrected}\|U_1-U_2\|=O_{hp}\left((N\vep)^{-2}\|Y_1\|\right)\,.
\end{equation}
The second claim of Lemma \ref{l.y1} implies that the right hand side above is $o_{hp}(1)$. The same implies that for $m=1,2$ and $1\le j,l\le k$,
\begin{equation}
\label{proof.l6.eq3}U_m(j,l)=(\theta_j-\theta_i)\one(j=l)+o_{hp}(1)\,,N\to\infty\,.
\end{equation}
In other words, as $N\to\infty$, $U_1$ and $U_2$ converge to $\diag(\theta_1-\theta_i,\ldots,\theta_k-\theta_i)$ w.h.p. Therefore,
\begin{equation}
\label{proof.l6.eq4}\lambda_i(U_m)=o_{hp}(1)\,,m=1,2\,.
\end{equation}

Let $\tilde U_m$, for $m=1,2,$ be the $(k-1)\times(k-1)$ matrix (recall that $k\ge2$) obtained by deleting the $i$-th row and the $i$-th column of $U_m$, and let $\tilde u_m$ be the $(k-1)\times1$ vector obtained from the $i$-th column of $U_m$ by deleting its $i$-th entry. It is worth recording, for possible future use, that
\begin{equation}
\label{proof.l6.eq10}\|\tilde u_m\|=o_{hp}(1)\,,m=1,2\,,
\end{equation}
which follows from \eqref{proof.l6.eq3}, and that
\begin{equation}
\label{proof.l6.eq11}\|\tilde u_1-\tilde u_2\|=O_{hp}\left((N\vep)^{-2}\|Y_1\|\right)\,,
\end{equation}
follows from \eqref{proof.l6.corrected}.

Equations \eqref{proof.l6.eq3} and \eqref{proof.l6.eq4} imply that $\tilde U_m-\lambda_i(U_m)I_{k-1}$ converges w.h.p.\ to
\[
\diag(\theta_1-\theta_i,\ldots,\theta_{i-1}-\theta_i,\theta_{i+1}-\theta_i,\theta_k-\theta_i)\,.
\]
Since $i\in\cal I$, the above matrix is invertible. Fix $\delta>0$ such that every matrix in the closed $\delta$-neighborhood $B_\delta$, in the sense of operator norm, of the above matrix is invertible. Let 
\begin{equation}
\label{proof.l6.c4}C_4=\sup_{E\in B_\delta}\|E^{-1}\|\,.
\end{equation}
Then, $C_4<\infty$. Besides, there exists $C_5<\infty$ satisfying
\begin{equation}
\label{proof.l6.c5}\left\|E_1^{-1}-E_2^{-1}\right\|\le C_5\|E_1-E_2\|\,,E_1,E_2\in B_\delta\,.
\end{equation}

Fix $N\ge1$ and a sample point such that  $\tilde U_m-\lambda_i(U_m)I_{k-1}$ belongs to $B_\delta$. Then, it is invertible. Define a $(k-1)\times1$ vector
\[
\tilde v_m=-\left[\tilde U_m-\lambda_i(U_m)I_{k-1}\right]^{-1}\tilde u_m\,,m=1,2\,,
\]
and a $k\times1$ vector
\[
v_m=\left[\tilde v_m(1),\ldots,\tilde v_m(i-1),\,1,\,\tilde v_m(i),\ldots,\tilde v_m(k-1)\right]^\prime\,,m=1,2\,.
\]
It is immediate that
\begin{equation}
\label{proof.l6.eq5}\|\tilde v_m\|\le C_4\|\tilde u_m\|\,,m=1,2\,.
\end{equation}

Our next claim is that
\begin{equation}
\label{proof.l6.eq6}U_mv_m=\lambda_i(U_m)v_m\,,m=1,2\,.
\end{equation}
This claim is equivalent to
\begin{equation}
\label{proof.l6.eq7}\left[U_m-\lambda_i(U_m)I_k\right]v_m=0\,.
\end{equation}
Let $\bar U_m$ be the $(k-1)\times k$ matrix obtained by deleting the $i$-th row of $U_m-\lambda_i(U_m)I_k$. Since the latter matrix is singular, and $\tilde U_m-\lambda_i(U_m)I_{k-1}$ is invertible, it follows that the $i$-th row of $U_m-\lambda_i(U_m)I_k$ lies in the row space of $\bar U_m$. In other words, the row spaces of $U_m-\lambda_i(U_m)I_k$ and $\bar U_m$ are the same, and so do their null spaces. Thus, \eqref{proof.l6.eq7} is equivalent to
\[
\bar U_mv_m=0\,.
\]
To see the above, observe that the $i$-th column of $\bar U_m$ is $\tilde u_m$, and hence we can partition
\[
\bar U_m=\left[\bar U_{m1}\,\tilde u_m\,\bar U_{m2}\right]\,,
\]
where $\bar U_{m1}$ and $\bar U_{m2}$ are of order $(k-1)\times(i-1)$ and $(k-1)\times(k-i)$, respectively. Furthermore,
\[
\left[\bar U_{m1}\,\bar U_{m2}\right]=\tilde U_m-\lambda_i(U_m)I_{k-1}\,.
\]
Therefore,
\[
\bar U_mv_m=\tilde u_m+\left[\bar U_{m1}\,\bar U_{m2}\right]\tilde v_m=\tilde u_m+\left(\tilde U_m-\lambda_i(U_m)I_{k-1}\right)\tilde v_m=0\,.
\]
Hence, \eqref{proof.l6.eq7} follows, which proves \eqref{proof.l6.eq6}. 

Next, we note
\begin{align*}
&\|v_1-v_2\|\\
=&\|\tilde v_1-\tilde v_2\|\\
\le&\left\|\left(\tilde U_1-\lambda_i(U_1)I_{k-1}\right)^{-1}\right\|\|\tilde u_1-\tilde u_2\|\\
&\,\,\,\,\,+\left\|\left(\tilde U_1-\lambda_i(U_1)I_{k-1}\right)^{-1}-\left(\tilde U_2-\lambda_i(U_2)I_{k-1}\right)^{-1}\right\|\|\tilde u_2\|\\
\le&C_4\|\tilde u_1-\tilde u_2\|+C_5\left\|\left(\tilde U_1-\lambda_i(U_1)I_{k-1}\right)-\left(\tilde U_2-\lambda_i(U_2)I_{k-1}\right)\right\|\|\tilde u_2\|\,,
\end{align*}
$C_4$ and $C_5$ being as in \eqref{proof.l6.c4} and \eqref{proof.l6.c5}, respectively. Recalling that the above calculation was done on an event of high probability, what we have proven, with the help of \eqref{proof.l6.corrected} and \eqref{proof.l6.eq11}, is that
\begin{align*}
\|v_1-v_2\|=O_{hp}\left((N\vep)^{-2}\|Y_1\|\right)\,.
\end{align*}
Furthermore, \eqref{proof.l6.eq10} and \eqref{proof.l6.eq5} imply that
\[
\|\tilde v_m\|=o_{hp}(1)\,.
\]

Finally, noting that 
\[
U_m(i,i)=0\,,m=1,2\,,
\]
and that
\[
v_m(i)=1\,,m=1,2\,,
\]
it follows that
\begin{align*}
&\left|\lambda_i(U_1)-\lambda_i(U_2)\right|\\
=&\left|\sum_{1\le j\le k,\,j\neq i}U_1(i,j)v_1(j)-\sum_{1\le j\le k,\,j\neq i}U_2(i,j)v_2(j)\right|\\
\le&\sum_{1\le j\le k,\,j\neq i}|U_1(i,j)||v_1(j)-v_2(j)|+\sum_{1\le j\le k,\,j\neq i}|U_1(i,j)-U_2(i,j)|v_2(j)|\\
=&O_{hp}\left(\|\tilde u_1\|\|v_1-v_2\|+\|U_1-U_2\|\|\tilde v_2\|\right)\\
=&o_{hp}\left((N\vep)^{-2}\|Y_1\|\right)\,.
\end{align*}
Recalling \eqref{proof.l6.u1} and \eqref{proof.l6.u2}, \eqref{proof.l6.eq2} follows, which completes the proof in conjunction with \eqref{proof.l6.eq1}.
\end{proof}

Now, we are in a position to prove Theorem \ref{t.secorder}.

\begin{proof}[Proof of Theorem \ref{t.secorder}]
Recalling that 
\[
Y_1(i,i)=\theta_iN\vep\,e_i^\prime We_i\,,
\]
it suffices to show that
\begin{equation}
\label{t.secorder.main}\mu-\E(\mu)=\mu^{-1}Y_1(i,i)+o_p(\sqrt\vep)\,.
\end{equation}
Lemma \ref{proof.l6} implies that 
\begin{align}
\nonumber\mu-\bar\mu=&\mu^{-1}Y_1(i,i)+o_{hp}\left((N\vep)^{-1}\|Y_1\|+\sqrt\vep\right)\\
\label{t.secorder.eq-1}=&O_{hp}\left((N\vep)^{-1}\|Y_1\|+\sqrt\vep\right)\,,
\end{align}
a consequence of which, combined with Lemma \ref{l.y1},  is that
\begin{equation}
\label{eq.boundmubar}\lim_{N\to\infty}(N\vep)^{-1}\bar\mu=\theta_i\,.
\end{equation}
Thus,
\begin{align}
\nonumber\left|\frac1{\bar\mu}Y_1(i,i)-\frac1{\mu}Y_1(i,i)\right|=&O_{hp}\left((N\vep)^{-2}|\mu-\bar\mu|\|Y_1\|\right)\\
\nonumber=&o_{hp}\left(|\mu-\bar\mu|\right)\\
\nonumber=&o_{hp}\left((N\vep)^{-1}\|Y_1\|+\sqrt\vep\right)\\
=&o_p(\sqrt\vep)\,,\label{t.secorder.eq0}
\end{align}
Lemma \ref{l.y1} implying the second line, the third line following from \eqref{t.secorder.eq-1} and the fact that
\begin{equation}
\label{eq.normY1}\|Y_1\|=O_p\left(N\vep^{3/2}\right)\,,
\end{equation}
which is also a consequence of the former lemma, being used for the last line. Using Lemma \ref{proof.l6} once again, we get that
\begin{equation}
\label{t.secorder.eq1}
\mu=\bar\mu+\frac1{\bar\mu}Y_1(i,i)+o_{hp}\left((N\vep)^{-1}\|Y_1\|+\sqrt\vep\right)\,.
\end{equation}

Let
\[
R=\mu-\bar\mu-\frac1{\bar\mu}Y_1(i,i)\,.
\]
Clearly,
\[
\E(R)=\E(\mu)-\bar\mu\,,
\]
and \eqref{t.secorder.eq1} implies that for $\delta>0$ there exists $\eta>1$ with
\begin{align*}
\E|R|
\le\delta\left(\sqrt\vep+(N\vep)^{-1}\E\|Y_1\|\right)+\E^{1/2}\left(\mu-\bar\mu-\frac1{\bar\mu}Y_1(i,i)\right)^2O\left(e^{-(\log N)^\eta}\right)\,.
\end{align*}
Lemma \ref{l.y1} implies that
\[
\E|R|\le o(\sqrt\vep)+\E^{1/2}\left(\mu-\bar\mu-\frac1{\bar\mu}Y_1(i,i)\right)^2O\left(e^{-(\log N)^\eta}\right)\,.
\]

Next, \eqref{eq.boundmubar} and that $|\mu|\le N^2$ a.s.\ imply that
\begin{align*}
\E^{1/2}\left(\mu-\bar\mu-\frac1{\bar\mu}Y_1(i,i)\right)^2
=&O\left(N^2\right)\\
=&o\left(\vep^{1/2}N^3\right)\\
=&o\left(\vep^{1/2}e^{(\log N)^\eta}\right)\,.
\end{align*}
Thus,
\[
\E|R|=o(\sqrt\vep)\,,
\]
and hence
\[
\E(\mu)=\bar\mu+o(\sqrt\vep)\,.
\]
This, in view of \eqref{t.secorder.eq1}, implies that
\begin{align*}
\mu=&\E(\mu)+\frac1{\bar\mu}Y_1(i,i)+o_p\left((N\vep)^{-1}\|Y_1\|+\sqrt\vep\right)\\
=&\E(\mu)+\frac1{\bar\mu}Y_1(i,i)+o_p\left(\sqrt\vep\right)\,,
\end{align*}
the second line following from \eqref{eq.normY1}. This establishes \eqref{t.secorder.main} with the help of \eqref{t.secorder.eq0}, and hence the proof.
\end{proof}

\subsection{Proof of Theorem \ref{t.main}}
Theorems \ref{t.inprob} and \ref{t.secorder} establish Theorem \ref{t.main} with the help of \eqref{eq.clt}. 

\subsection{Proof of Theorem \ref{t.mean}}
Now we shall proceed toward proving Theorem \ref{t.mean}. For the rest of this section, that is, this subsection and the subsequent two, Assumption F2.\ holds. In other words, $r_1,\ldots,r_k$ are assumed to be Lipschitz continuous and hence so is $f$.

The following lemma essentially proves Theorem \ref{t.mean}.

\begin{lemma}\label{proof.approx}
Under Assumptions E1.\ and F2.,
\[
\mu=\lambda_i\left(Y_0+(N\vep\theta_i)^{-2}\E(Y_2)\right)+O_p\left(\sqrt\vep+(N\vep)^{-1}\right)\,.
\]
\end{lemma}

\begin{proof}
Lemma \ref{proof.l3} implies that
\[
\mu=\lambda_i\left(\sum_{n=0}^3\mu^{-n}\E(Y_n)\right)+O_p\left(\mu^{-1}\|Y_1\|+\sum_{n=4}^L\mu^{-n}\|\E(Y_n)\|\right)+o_p(\sqrt\vep)\,.
\]
Equation \eqref{eq.normY1} implies that
\[
\mu=\lambda_i\left(\sum_{n=0}^3\mu^{-n}\E(Y_n)\right)+O_p\left(\sqrt\vep+\sum_{n=4}^L\mu^{-n}\|\E(Y_n)\|\right)\,.
\]
From \eqref{proof.l4.eq5},  it follows that
\[
\sum_{n=4}^L\mu^{-n}\|\E(Y_n)\|=O_p\left((N\vep)^{-1}\right)\,,
\]
and hence
\begin{equation}
\label{proof.approx.eq0}\mu=\lambda_i\left(\sum_{n=0}^3\mu^{-n}\E(Y_n)\right)+O_p\left(\sqrt\vep+(N\vep)^{-1}\right)\,.
\end{equation}

Lemma \ref{est.l4}, in particular \eqref{est.l4.eq2} therein, implies that
\[
\|\E(Y_3)\|=O\left((N\vep)^2\right)\,,
\]
and hence
\[
\mu^{-3}\|\E(Y_3)\|=O_p\left((N\vep)^{-1}\right)\,.
\]
This, in conjunction with \eqref{proof.approx.eq0}, implies that
\begin{equation}
\label{proof.approx.eq1}\mu=\lambda_i\left(Y_0+\mu^{-2}\E(Y_2)\right)+O_p\left(\sqrt\vep+(N\vep)^{-1}\right)\,.
\end{equation}

An immediate consequence of the above and \eqref{proof.l4.eq5} is that
\begin{align}\label{proof.approx.eq2}
\mu=\lambda_i(Y_0)+O_p(1)\,.
\end{align}

Applying Fact \ref{f.gct} as in the proof of Lemma \ref{proof.l1}, it can be shown that
\begin{equation}
\label{proof.approx.eq3}\left|\lambda_i(Y_0)-Y_0(i,i)\right|\le\sum_{1\le j\le k,\,j\neq i}|Y_0(i,j)|\,.
\end{equation}
Since $r_i$ and $r_j$ are Lipschitz functions, it holds that
\[
e_i^\prime e_j=\one(i=j)+O\left(N^{-1}\right)\,.
\]
Hence, it follows that
\[
Y_0(i,i)=N\vep\left(\theta_i+O(N^{-1})\right)=N\vep\theta_i+O(\vep)\,,
\]
and similarly,
\[
Y_0(i,j)=O(\vep)\,,j\neq i\,.
\]
Combining these findings with \eqref{proof.approx.eq3} yields that
\begin{equation}
\label{proof.approx.eq4}\lambda_i(Y_0)=N\vep\theta_i+O(\vep)\,.
\end{equation}

Equations \eqref{proof.approx.eq2} and \eqref{proof.approx.eq4} together imply that
\begin{equation}
\label{proof.approx.eq5}\mu=N\vep\theta_i+O_p(1)\,.
\end{equation}
Therefore,
\begin{align*}
&\left\|\mu^{-2}\E(Y_2)-(N\vep\theta_i)^{-2}\E(Y_2)\right\|\\
=&O_p\left((N\vep)^{-3}\|\E(Y_2)\|\right)\\
=&O_p\left((N\vep)^{-1}\right)\,.
\end{align*}
This in conjunction with \eqref{proof.approx.eq1} completes the proof.
\end{proof}

Theorem \ref{t.mean} is a simple corollary of the above lemma, as shown below.

\begin{proof}[Proof of Theorem \ref{t.mean}]
A consequence of Theorem \ref{t.secorder} is that
\[
\mu-\E(\mu)=O_p(\sqrt\vep)\,.
\]
The claim of Lemma \ref{proof.approx} is equivalent to
\[
\lambda_i(B)-\mu=O_p\left(\sqrt\vep+(N\vep)^{-1}\right)\,.
\]
The proof follows by adding the two equations, and noting that $B$ is a deterministic matrix.
\end{proof}

\subsection{Proof of Theorem \ref{t.eigvec}}
Next we proceed towards the proof of Theorem \ref{t.eigvec}, for which the following lemma will be useful. 

\begin{lemma}\label{proof.l8}
Under Assumptions E1.\ and F2., as $N\to\infty$,
\[
e_j^\prime\left(I-\mu^{-1}W\right)^{-n}e_l=\one(j=l)+O_p\left((N\vep)^{-1}\right)\,,1\le j,l\le k\,,n=1,2\,.
\]
\end{lemma}

\begin{proof}
For a fixed $n=1,2,$ expand
\[
\left(I-\mu^{-1}W\right)^{-n}=
I+n\mu^{-1}W+O_p\left(\mu^{-2}\|W\|^2\right)\,.
\]
The proof can be completed by proceeding along similar lines as in the proof of Lemma \ref{proof.approx}.
\end{proof}

Now we are in a position to prove Theorem \ref{t.eigvec}.

\begin{proof}[Proof of Theorem \ref{t.eigvec}]
Theorem \ref{t.inprob} implies that \eqref{t.eigvec.claim1} holds for any $i\in\cal I$. Fix such an $i$, denote
\[
\mu=\lambda_i(A)\,,
\] 
and let $v$ be an eigenvector of $A$, having norm $1$, corresponding to $\mu$ satisfying \eqref{t.eigenvec.align}.  In \eqref{t.eigenvec.eq3} below it is shown that $v$ is uniquely determined with high probability.

Fix $k\ge2$, and $j\in\{1,\ldots,k\}\setminus\{i\}$.  Premultiplying \eqref{eq.master} by $e_j^\prime$ yields that
\begin{equation}\label{t.eigvec.eq1}
e_j^\prime v=N\vep\sum_{l=1}^k\theta_l(e_l^\prime v)e_j^\prime\left(\mu I-W\right)^{-1}e_l\,,\text{ w.h.p.}
\end{equation}
Therefore,
\begin{align*}
\nonumber&e_j^\prime v\left(1-\theta_j\frac{N\vep}\mu e_j^\prime\left(I-\mu^{-1}W\right)^{-1}e_j\right)\\
=&\frac{N\vep}\mu\sum_{1\le l\le k,\,l\neq j}\theta_l(e_l^\prime v)e_j^\prime\left(I-\mu^{-1}W\right)^{-1}e_l\,,\text{ w.h.p.}
\end{align*}
Lemma \ref{proof.l8} implies that as $N\to\infty$, 
\[
1-\theta_j\frac{N\vep}\mu e_j^\prime\left(I-\mu^{-1}W\right)^{-1}e_j\prob1-\frac{\theta_j}{\theta_i}\neq0\,.
\]
Therefore,
\begin{align*}
e_j^\prime v=&O_p\left(\frac{N\vep}\mu\sum_{1\le l\le k,\,l\neq j}\theta_l(e_l^\prime v)e_j^\prime\left(I-\mu^{-1}W\right)^{-1}e_l\right)\\
=&O_p\left(\sum_{1\le l\le k,\,l\neq j}\left|e_j^\prime\left(I-\mu^{-1}W\right)^{-1}e_l\right|\right)\\
=&O_p\left((N\vep)^{-1}\right)\,,
\end{align*}
the last line being another consequence of Lemma \ref{proof.l8}. Thus, \eqref{t.eigvec.claim3} holds. 

Premultiplying \eqref{eq.master} with its own transpose and using the fact $v'v=1$, we get that
\[
1=(N\vep)^2\sum_{l,m=1}^k\theta_l\theta_m(e_l^\prime v)(e_m^\prime v)e_l^\prime\left(\mu I-W\right)^{-2}e_m\,,
\]
that is,
\begin{align}
\label{t.eigvec.eq2}&\theta_i^2(e_i^\prime v)^2e_i^\prime\left(I-\mu^{-1}W\right)^{-2}e_i\\
\nonumber=&(N\vep)^{-2}\mu^2-\sum_{(l,m)\in\{1,\ldots,k\}^2\setminus\{(i,i)\}}\theta_l\theta_m(e_l^\prime v)(e_m^\prime v)e_l^\prime\left(I-\mu^{-1} W\right)^{-2}e_m\,.
\end{align}

Using Lemma \ref{proof.l8} once again, it follows that 
\[
e_i^\prime\left(I-\mu^{-1}W\right)^{-2}e_i=1+O_p\left((N\vep)^{-1}\right)\,.
\]
As $N\to\infty$, the right hand side of \eqref{t.eigvec.eq2} converges in probability to $\theta_i^2$. This along with \eqref{t.eigenvec.align} shows that
\begin{equation}\label{t.eigenvec.eq3}
e_i^\prime v\prob1\,.
\end{equation}
An immediate consequence of \eqref{t.eigvec.claim1} and the above is that the eigenvector $v$ satisfying \eqref{t.eigenvec.align} is uniquely determined with high probability.  It should be noted that if the inequality in \eqref{t.eigenvec.align} were reversed, then $-1$ would have been the limit in probability in \eqref{t.eigenvec.eq3}.

In view of \eqref{t.eigvec.eq2} and \eqref{t.eigenvec.eq3},  \eqref{t.eigvec.claim2} would follow once it's shown that 
\begin{equation}
\label{t.eigvec.eq4}(N\vep)^{-2}\mu^2=\theta_i^2+O_p\left((N\vep)^{-1}\right)\,,
\end{equation}
and that for all $(l,m)\in\{1,\ldots,k\}^2\setminus\{(i,i)\}$,
\begin{equation}
\label{t.eigvec.eq5}(e_l^\prime v)(e_m^\prime v)e_l^\prime\left(I-\mu^{-1} W\right)^{-2}e_m=O_p\left((N\vep)^{-1}\right)\,.
\end{equation}

Equation \eqref{t.eigvec.eq4} is a trivial consequence of \eqref{proof.approx.eq5}. For \eqref{t.eigvec.eq5}, assuming without loss of generality that $l\neq i$, \eqref{t.eigvec.claim3} implies that
\begin{align*}
\left|(e_l^\prime v)(e_m^\prime v)e_l^\prime\left(I-\mu^{-1} W\right)^{-2}e_m\right|=&\left|(e_m^\prime v)e_l^\prime\left(I-\mu^{-1} W\right)^{-2}e_m\right|O_p\left((N\vep)^{-1}\right)\\
\le&\left|e_l^\prime\left(I-\mu^{-1} W\right)^{-2}e_m\right|O_p\left((N\vep)^{-1}\right)\\
=&O_p\left((N\vep)^{-1}\right)\,,
\end{align*}
the last line following from Lemma \ref{proof.l8}. Thus, \eqref{t.eigvec.eq5} follows, which in conjunction with \eqref{t.eigvec.eq4} establishes \eqref{t.eigvec.claim2}. This completes the proof.
\end{proof}

\subsection{Proof of Theorem \ref{t.flac}}
Finally, Theorem \ref{t.flac} is proved below, based on Assumptions E2.\ and F2.

\begin{proof}[Proof of Theorem \ref{t.flac}] 
Fix $i\in\cal I$. Recall \eqref{proof.l1.eqnew1} and \eqref{proof.l1.eqnew2}, and let $u$ be as defined in the former. Let $\tilde u$ be the column vector obtained by deleting the $i$-th entry of $u$, $\tilde V_i$ be the column vector obtained by deleting the $i$-th entry of the $i$-th column of $V$, and $\tilde V$ be the $(k-1)\times(k-1)$ matrix obtained by deleting the $i$-th row and $i$-th column of $V$. Then, \eqref{proof.l1.eqnew2} implies that
\begin{equation}
\label{t.flac.eq3}
\mu\tilde u=\tilde V\tilde u+u(i)\tilde V_i\,,\text{ w.h.p.}
\end{equation}
Lemma \ref{proof.l0} implies that 
\[
\left\|I_{k}-\mu^{-1}V-\diag\left(1-\frac{\theta_1}{\theta_i},\ldots,1-\frac{\theta_k}{\theta_i}\right)\right\|=o_{hp}(1)\,,
\]
and hence $I_{k-1}-\mu^{-1}\tilde V$ is non-singular w.h.p. Thus, \eqref{t.flac.eq3} implies that
\begin{equation}
\label{t.flac.eq2}\tilde u=u(i)\mu^{-1}\left(I_{k-1}-\mu^{-1}\tilde V\right)^{-1}\tilde V_i\,,\text{ w.h.p.}
\end{equation}

The next step is to show that
\begin{equation}
\label{t.flac.eq4}\left\|\mu^{-1}V-\diag\left(\frac{\theta_1}{\theta_i},\ldots,\frac{\theta_k}{\theta_i}\right)\right\|=o_{p}\left(\sqrt\vep\right)\,.
\end{equation}
To see this, use the fact that $f$ is Lipschitz to write for a fixed $1\le j,l\le k$,
\begin{align}
\nonumber V(j,l)=&N\vep\sqrt{\theta_j\theta_l}\left(e_j^\prime e_l+\mu^{-1}e_j^\prime We_l+O_p\left(\mu^{-2}\|W\|^2\right)\right)\\
\nonumber =&N\vep\sqrt{\theta_j\theta_l}\left(e_j^\prime e_l+O_p\left((N\vep)^{-1}\right)\right)\\
\nonumber =&N\vep\theta_j\left(\one(j=l)+O_p\left((N\vep)^{-1}\right)\right)\\
\label{t.flac.eqnew}=&N\vep\theta_j\left(\one(j=l)+o_p\left(\sqrt\vep\right)\right)\,,
\end{align}
the last line following from the fact that 
\begin{equation}
\label{t.flac.eq0}(N\vep)^{-1}=o\left(\sqrt\vep\right)\,,
\end{equation}
which is a consequence of \eqref{t.flac.eq1}. This along with  \eqref{proof.approx.eq5} implies that
\begin{equation}
\label{t.flac.eq6}(N\vep\theta_i)^{-1}\mu=1+o_p\left(\sqrt\vep\right)\,.
\end{equation}
Combining this with \eqref{t.flac.eqnew} yields that
\[
\mu^{-1}V(j,l)=\theta_i^{-1}\theta_j\one(j=l)+o_p\left(\sqrt\vep\right)\,.
\]
Thus, \eqref{t.flac.eq4} follows, an immediate consequence of which is that
\begin{equation}
\label{t.flac.eq5}\left\|\left(I_{k-1}-\mu^{-1}\tilde V\right)^{-1}-\tilde D\right\|=o_p\left(\sqrt\vep\right)\,,
\end{equation}
where
\[
\tilde D=\left[\diag\left(1-\frac{\theta_1}{\theta_i},\ldots,1-\frac{\theta_{i-1}}{\theta_i},1-\frac{\theta_{i+1}}{\theta_i},\ldots,1-\frac{\theta_k}{\theta_i}\right)\right]^{-1}\,.
\]

Next, fix $j\in\{1,\ldots,k\}\setminus\{i\}$. By similar arguments as above, it follows that
\begin{align*}
V(i,j)=&N\vep\sqrt{\theta_i\theta_j}\left(\sum_{n=0}^3\mu^{-n}e_i^\prime W^ne_j+O_p\left(\mu^{-4}\|W\|^4\right)\right)\\
=&N\vep\sqrt{\theta_i\theta_j}\sum_{n=0}^3\mu^{-n}e_i^\prime W^ne_j+O_p\left((N\vep)^{-1}\right)\\
=&N\vep\sqrt{\theta_i\theta_j}\sum_{n=1}^2\mu^{-n}e_i^\prime W^ne_j+o_p\left(\sqrt\vep\right)\,,
\end{align*}
using \eqref{t.flac.eq0} once again, because
\[
N\vep e_i^\prime e_j=O(\vep)=o\left(\sqrt\vep\right)\,,
\]
and 
\[
N\vep\mu^{-3} e_i^\prime W^3e_j=O_p\left((N\vep)^{-2}\E(e_i^\prime W^3e_j)\right)=o_p\left(\sqrt\vep\right)\,,
\]
by  \eqref{est.l4.eq2}. Thus,
\begin{align*}
V(i,j)-N\vep\sqrt{\theta_i\theta_j}\mu^{-1}e_i^\prime We_j
=&N\vep\sqrt{\theta_i\theta_j}\mu^{-2}e_i^\prime W^2e_j+o_p\left(\sqrt\vep\right)\\
=&N\vep\sqrt{\theta_i\theta_j}\mu^{-2}\E\left(e_i^\prime W^2e_j\right)+o_p\left(\sqrt\vep\right)\\
=&(N\vep)^{-1}\theta_j^{1/2}\theta_i^{-3/2}\E\left(e_i^\prime W^2e_j\right)+o_p\left(\sqrt\vep\right)\,,
\end{align*}
the second line following from Lemma \ref{est.l3}, and the last line from \eqref{t.flac.eq0}, \eqref{t.flac.eq6} and Lemma \ref{est.l2}. 
In particular,
\[
V(i,j)=O_p(1)\,.
\]

The above in conjunction with \eqref{t.flac.eq5} implies that 
\begin{align*}
&\left[\left(I_{k-1}-\mu^{-1}\tilde V\right)^{-1}\tilde V_i\right](j)\\
=&\left(1-\frac{\theta_j}{\theta_i}\right)^{-1}\sqrt{\theta_i\theta_j}\left[(N\vep)^{-1}\theta_i^{-2}\E\left(e_i^\prime W^2e_j\right)+N\vep\mu^{-1}e_i^\prime We_j\right]+o_p(\sqrt\vep)\,.
\end{align*}

In light of \eqref{t.flac.eq2}, the above means that
\begin{align*}
&e_j^\prime v\\
=&(e_i^\prime v)\mu^{-1}\left(1-\frac{\theta_j}{\theta_i}\right)^{-1}\left[(N\vep)^{-1}\theta_i^{-1}\E\left(e_i^\prime W^2e_j\right)+N\vep\theta_i\mu^{-1}e_i^\prime We_j+o_p(\sqrt\vep)\right]\\
=&\mu^{-1}\left(1-\frac{\theta_j}{\theta_i}\right)^{-1}\left[(N\vep)^{-1}\theta_i^{-1}\E\left(e_i^\prime W^2e_j\right)+N\vep\theta_i\mu^{-1}e_i^\prime We_j+o_p(\sqrt\vep)\right]\,,
\end{align*}
the last line following from \eqref{t.eigvec.claim2} and \eqref{t.flac.eq0}. Using \eqref{t.flac.eq6} once again yields that
\[
N\vep(e_j^\prime v)=\frac1{\theta_i-\theta_j}\left[(N\vep)^{-1}\theta_i^{-1}\E\left(e_i^\prime W^2e_j\right)+N\vep\theta_i\mu^{-1}e_i^\prime We_j\right]+o_p(\sqrt\vep)\,.
\]
This completes the proof.
\end{proof}

\section{Appendix}\label{sec:app} 

\begin{lemma}\label{lemma:RIeigenfunctions}
The eigenfunctions $\{r_i:1\le i\le k\}$ of the operator $I_f$ are Riemann integrable.
\end{lemma}
\begin{proof}
Let $D_f\subset [0,1]\times [0,1]$ be the set of discontinuity points $f$. Since $f$ is Riemann integrable,  the  Lebesgue measure of $D_f$ is $0$. Let
$$D_f^x=\{y\in [0,\,1]:\, (x,y)\in D_f\}, \, x\in [0,\,1]\,.$$
If $\lambda$ is the one dimensional Lebesgue measure, then Fubini's theorem implies that
$$E=\{ x\in [0,1]: \lambda(D_f^x)=0\}$$
has full measure. Fix an $x\in E$ and consider $x_n\to x$ and observe that $$f(x_n,y)\to f(x,y) \text{ for all $y\notin D_f^x$}.$$ 
Fix $1\le i\le k$ and let $\theta_i$ be the eigenvalue with corresponding eigenfunction $r_i$, that is,
\begin{equation}\label{eq:eigeneq}
r_i(x)= \frac{1}{\theta_i}\int_0^1 f(x,y) r_i(y) \, dy.
\end{equation}
Using $f$ is bounded and $r\in L^2[0,1]$, dominated convergence theorem implies
$$r_i(x_n)= \frac{1}{\theta_i}\int_{(D_f^x)^c} f(x_n,y) r_i(y)\,  dy\to \frac{1}{\theta_i}\int_0^1 f(x,y) r_i(y)\, dy=r_i(x)$$
and hence $r$ is continuous at $x$. So the discontinuity points of $r_i$ form a subset of $E^c$ which has Lebesgue measure $0$. Further, \eqref{eq:eigeneq} shows that $r_i$ is bounded and hence Riemann integrability follows. 
\end{proof}

The following result is a version of the Perron-Frobenius theorem in the infinite dimensional setting (also known as the Krein-Rutman theorem). Since our integral operator is positive, self-adjoint and finite dimensional so the proof in this setting is much simpler and can be derived following the work of \cite{ninio:1976}. In what follows, we use for $f, g\in L^2[0,1]$, the inner product 
$$\langle f, \, g\rangle =\int_0^1 f(x)g(x)dx.$$

\begin{lemma}\label{lemma:KRT}
Suppose $f>0$ a.e.\ on $[0,1]\times[0,1]$. Then largest eigenvalue $\theta_1$ of $T_f$ is positive and the corresponding eigenfunction $r_1$ can be chosen such that $r_1(x)>0$ for almost every $x\in [0,1]$. Further, $\theta_1>\theta_2$.
\end{lemma}

\begin{proof}
First observe that 
\begin{align*}
 0<\theta_1= \langle r_1,\, \theta_1 r_1\rangle & = \langle r_1, \, I_f(r_1)\rangle= | \langle r_1, \, I_f(r_1)\rangle|\\
 &\le \langle u_1, I_f(u_1)\rangle \le \theta_1
\end{align*}
where $u_1(x)=|r_1|(x)$ and the last inequality follows from the Rayleigh-Ritz formulation of the largest eigenvalue. Hence note that the string of inequalities is actually an equality, that is,
$$ \langle r_1, \, I_f(r_1) \rangle= \langle u_1, I_f(u_1)\rangle.$$
Breaking $r_1= r_1^{+}-r_1^{-}$ implies either $r_1^+=0$ or $r_1^{-}=0$ almost everywhere. 
Without loss of generality assume that $r_1\ge 0$ almost everywhere. Using 
$$\theta_1 r_1(x)= \int_0^1 f(x,y) r_1(y)\, dy$$
Note that if $r_1(x)$ is zero for some $x$ then due to the positivity assumption on $f$, $r_1(y)=0$ for almost every $y\in [0,1]$ which is a contradiction. Hence we have that $r_1(x)>0$ almost every $x\in [0, 1]$.

For the final claim, without loss of generality assume that $\int_0^1 r_1(x)\, dx\ge 0$. If $\theta_1=\theta_2$, then the previous argument would give us $r_2(x)>0$ and this will contradict the orthogonality of $r_1$ and $r_2$. 
\end{proof}

Lemmas \ref{est.l1} -- \ref{est.l4} are proved in the rest of this section. Therefore, the notations used here should refer to those in Section \ref{sec:est} and should not be confused with those in Section \ref{sec:proof}. For example, $e_1$ and $e_2$ are as in Lemma \ref{est.l2}.

\begin{proof}[Proof of Lemma~\ref{est.l1}]
Note that for any even integer $k$ 
\begin{equation}\label{eq:spn:trace}
\E(\|W\|^k) \le \E(\Tr (W^k)).
\end{equation}
Using $\E(W(i,j)^2)\le \vep M$ and condition~\eqref{eq.vep} it is immediate that conditions of Theorem 1.4 of \cite{vu2007spec} are satisfied. We shall use the following estimate from the proof of that result. It follows from \cite[Section 4]{vu2007spec}
\begin{equation} \label{est.l1p.1}
\E(\Tr (W^k)) \le K_1N (2\sqrt{\vep MN})^k
\end{equation}
where $K_1$ is some positive constant and there exists a constant $a>0$ such that $k$ can be chosen as
$$k = \sqrt{2} a (\vep M)^{1/4} N^{1/4}.$$ 

Using~\eqref{eq:spn:trace}, \eqref{est.l1p.1} and $(1-x)^k \le e^{-kx}$ for $k$, $x > 0$, 
\begin{align} 
&P \left( \| W \| \ge 2 \sqrt{MN\vep} + C_1 (N \vep )^{1/4} ( \log N)^{\xi/4}\right)\nonumber \\
&= K_1 N \left( 1 - \frac{C_1 (N \vep )^{1/4} ( \log N)^{\xi/4}}{ 2 \sqrt{MN\vep} + C_1 (N \vep )^{1/4} ( \log N)^{\xi/4}} \right)^k\nonumber \\
&\le K_1 N \exp\left(-\frac{k C_1 (N \vep )^{1/4} ( \log N)^{\xi/4}}{ 2 \sqrt{MN\vep} + C_1 (N \vep )^{1/4} ( \log N)^{\xi/4}} \right). \label{est.l1p.2}
\end{align}
Now plugging in the value of $k$ in the bound \eqref{est.l1p.2} and using
\[
2 \sqrt{M} + C_1 (N \vep )^{-1/4} ( \log N)^{\xi/4} \le 2 \sqrt{M} +C_1
\]
we have
\[
\eqref{est.l1p.2} \le K_1 N \exp\left(-\frac{C_1 a M^{1/4}\sqrt{2} ( \log N)^{\xi/4}}{ 2 \sqrt{M} + C_1} \right) \le e^{-C_2 (\log N )^{\xi/4}}
\]
for some constant $C_2>0$ and $N$ large enough. This proves~\eqref{eq:est.l1.1} and hence the lemma. 

\end{proof}
\begin{proof}[Proof of Lemma \ref{est.l2}]
Let $A$ be the event where Lemma \ref{est.l1} holds, that is, $\|W\| \le C \sqrt{N\vep}$ for some constant $C$. Since the entries of $e_1$ and $e_2$ are in $[-1/\sqrt{N}, 1/\sqrt{N}]$ so $\|e_i\|\le 1$ for $i=1,2$. Hence on the high probability event it holds that
\[
\left|\E\left(e_1^\prime W^ne_2 \textbf{1}_{A}\right)\right|\le (CN\vep)^{n/2}.
\]
We show that the above expectation on the low probability event $A^c$ is negligible. For that first observe
$$ |\E[(e_1^\prime W^n e_2)^2]|\le N^{nC^\prime}$$
for some constant $0<C^\prime<\infty$. Thus using Lemma~\ref{est.l1} one has
\begin{align*}
 \left|\E\left(e_1^\prime W^ne_2 \textbf{1}_{A^c}\right)\right|&\le \left|\E\left[(e_1^\prime W^ne_2)^2\right]^{1/2}\right|P(A_N^c)^{1/2} \\
 &\le \exp\left(nC^{\prime}\log N- 2^{-1}C_2 (\log N)^{\xi/4}\right)
\end{align*}
Since $n\le\log N$ and $\xi>8$ the result follows.
\end{proof}

\begin{proof}[Proof of Lemma~\ref{est.l3}]
The proof is similar to the proof of Lemma~6.5 of \cite{Erdos1}. The exponent in the exponential decay is crucial, so the proof is briefly sketched. 
Observe that
\begin{align}
&e_1^\prime W^ne_2-\E\left(e_1^\prime W^ne_2\right)\nonumber\\
=&\sum_{i\in \{1,\ldots, N\}^{n+1}}e_1(i_1)e_2(i_{n+1}) \left(\prod_{l=1}^{n}W(i_l, i_{l+1})-\E\left[\prod_{l=1}^{n}W(i_l, i_{l+1})\right]\right)\label{eq:centered}
\end{align}

To use the independence, one can  split the matrix $W$ as $W^{\prime} + W^{\pp}$ where the upper triangular matrix $W^{\prime}$ has entries $W^{\prime}(i,j) = W(i,j) \textbf{1}(i\leqslant j)$ and the lower triangular matrix $W^{\pp}$ with entries $W^{\pp}(i,j) = W(i,j) \textbf{1}(i > j)$. Therefore the above quantity under the sum breaks into $2^n$ terms each having similar properties. Denote one such term as
\[
L_n= \sum_{i\in \{1,\ldots, N\}^{n+1}}e_1(i_1)e_2(i_{n+1}) \left(\prod_{l=1}^{n}W^\prime(i_l, i_{l+1})-\E\left[\prod_{l=1}^{n}W^\prime(i_l, i_{l+1})\right]\right). 
\]
Using the fact that each entry of $e_1$ and $e_2$ are bounded by $1/\sqrt{N}$, it follows by imitating the proof of Lemma 6.5 of \cite{Erdos1} that
$$\E[|L_n|^p] \le \frac{\left(Cnp\right)^{np}(N\vep)^{np/2}}{N^{p/2}},$$
where $p$ is an even integer and $C$ is a positive constant, independent of $n$ and $p$. Rest of the $2^n-1$ terms arising in~\eqref{eq:centered} have the same bound and hence
\begin{align*}
    &P\left(\left|e_1^\prime W^ne_2-\E\left(e_1^\prime W^ne_2\right)\right|>N^{(n-1)/2}\vep^{n/2}(\log N)^{n\xi/4}\right)\\
    &\le \frac{\left(2Cnp\right)^{np}(N\vep)^{np/2}}{N^{p/2}N^{p(n-1)/2}\vep^{pn/2}(\log N)^{pn\xi/4}}= \frac{\left(2Cnp\right)^{np}}{(\log N)^{pn\xi/4}}.
\end{align*}
Choose $\eta\in (1,\, \xi/4)$ and consider 
\[
p=\frac{(\log N)^{\eta}}{2Cn}\,,
\]
(with $N$ large enough to make $p$ an even integer) to get
\begin{align*}
&P\left(\left|e_1^\prime W^ne_2-\E\left(e_1^\prime W^ne_2\right)\right|>N^{(n-1)/2}\vep^{n/2}(\log N)^{n\xi/4}\right)\\
&\le \exp\left(-\frac{1}{2C}(\log N)^{\eta}(\frac{\xi}{4}-\eta)\log \log N\right).
\end{align*}
 Note that $n\le L$, ensures that $p>1$. Since the bound is uniform over all $2\le n\le L$, the first bound~\eqref{est.l3.1} follows. 
 
 For \eqref{est.l3.2} one can use Hoeffding's inequality \cite[Theorem 2]{hoeffding:1963} as follows.
 
 Define
\[
\widetilde A(k,l)= A(k,l) e_1(k)e_2(l), \,\,\, 1\le k\le l\le N.
\]
Since $A(k,l)$ are Bernoulli random variables, so one has $\{\widetilde A(k,l): 1\le k\le l\le N\}$ are independent random variables taking values in  $[-1/N, 1/N]$ and hence by Hoeffding's inequality we have, for any $\delta>0$,
\begin{align*}
&P\left( \left|\sum_{1\le k\le l\le N} \widetilde A(k,l)- E\left(\sum_{1\le k\le l\le N} \widetilde A(k,l)\right)\right|>\delta N\vep\right)\\
&\le 2\exp\left(-\delta^2(N\vep)^2\right)\le 2 \exp\left(-\delta^2 (\log N)^{2\xi}\right).
\end{align*}
Dealing with the case $k>l$ similarly, the desired bound on $e_1^\prime We_2$ follows. 
\end{proof}

\begin{proof}[Proof of Lemma \ref{est.l4}]
Follows by a simple moment calculation.
\end{proof}

\section*{Acknowledgment}
The authors thank an anonymous referee for insightful comments that helped improve the paper significantly. RSH thanks Kavita Ramanan for a fruitful discussion. The research of AC and RSH was supported by the MATRICS grant of SERB. SC thanks Matteo Sfragara for helpful discussions.
The authors are grateful to Anirvan Chakraborty for pointing out that \eqref{t.eigenvec.align} is needed in the assumption, which they had missed in the previous versions.


\end{document}